\documentclass[ leqno, final]{article}
\usepackage{amsfonts,amsmath,amsthm}
\usepackage{epsfig}
\usepackage{graphicx}
\usepackage{showkeys}
\usepackage{color}

\def\R{\hbox{\bf R}}

\def\I{{\cal I}}

\def\a{\alpha}
\def\C{{\mathbb C}}

\def\eps{\varepsilon}

\def \buor {(\bar u_0)_r}
\def \buo {\bar u_0}

\renewcommand{\thesubsection}{\arabic{section}.\arabic{subsection}}

\renewcommand{\theequation}{\arabic{section}.\arabic{equation}}

\newcommand{\ba}{\begin{eqnarray}}
\newcommand{\ea}{\end{eqnarray}}

\newtheorem{theo}{\bf Theorem}[section]
\newtheorem{lem}[theo]{\bf Lemma}
\newtheorem{pro}[theo]{\bf Proposition}

\newtheorem{defi}[theo]{\bf Definition}
\newtheorem{rem}[theo]{\bf Remark}

\newtheorem{claim}[theo]{\bf Claim}

\renewcommand{\R}{{\mathbb R}}

\newcommand{\e}{\varepsilon}

\newenvironment{Proofc}[1]{\smallskip\par\noindent\textsc{#1}\quad}%
 {\hfill$\Box$\bigskip\par}

\begin{document}

\title{\bf Uniqueness and existence of spirals moving by forced mean curvature motion}
\author{\renewcommand{\thefootnote}{\arabic{footnote}}
N. Forcadel\footnotemark[1], C. Imbert\footnotemark[1] \, and R. Monneau\footnotemark[2]}

\footnotetext[1]{Universit\'e Paris-Dauphine, CEREMADE, UMR CNRS 7534,
place de Lattre de Tassigny, 75775 Paris cedex 16}
\footnotetext[2]{Universit\'e Paris-Est, CERMICS, Ecole des Ponts ParisTech, 
6-8 avenue Blaise Pascal, 77455 Marne-la-Vall\'ee Cedex  2, France}
\maketitle

\vspace{20pt}


\begin{abstract}
  In this paper, we study the motion of spirals by mean curvature type
  motion in the (two dimensional) plane. Our motivation comes from
  dislocation dynamics; in this context, spirals appear when a screw
  dislocation line reaches the surface of a crystal. The first main
  result of this paper is a comparison principle for the corresponding
  parabolic quasi-linear equation. As far as motion of spirals are
  concerned, the novelty and originality of our setting and results
  come from the fact that, first, the singularity generated by the
  attached end point of spirals is taken into account for the first
  time, and second, spirals are studied in the whole space. Our second
  main result states that the Cauchy problem is well-posed in the
  class of sub-linear weak (viscosity) solutions.  We also explain how
  to get the existence of smooth solutions when initial data satisfy
  an additional compatibility condition.
\end{abstract}



\section{Introduction}

In this paper we are interested in curves $(\Gamma_t)_{t>0}$ in $\R^2$
which are half lines with an end point attached at the origin.  These
lines are assumed to move with normal velocity
\begin{equation}\label{eq:law}
V_n= c+ \kappa
\end{equation}
where $\kappa$ is the curvature of the line and {\color{black} $c
  \in\R$} is a given constant. We will see that this problem reduces
to the study of the following quasi-linear parabolic equation in
non-divergence form
\begin{equation}\label{eq::4a}
 \displaystyle{r\bar{u}_t=c\sqrt{1 + r^2\bar{u}_r^2} 
   + \bar{u}_r \left(\frac{2+ r^2 \bar{u}_r^2}{1+ r^2 \bar{u}_r^2}\right)
   + \frac{r\bar{u}_{rr}}{1 + r^2\bar{u}_r^2}},
 \qquad t>0,r>0
\end{equation}
This paper is devoted to the proof of a comparison principle in the
class of sub-linear weak (viscosity) solutions and to the study of the
associated Cauchy problem.

\subsection{Motivations and known results}

\paragraph{Continuum mechanics.} From the viewpoint of
applications, the question of defining the motion of spirals in a two
dimensional space is motivated by the seminal paper of Burton, Cabrera
and Frank \cite{bcf51} where the growth of crystals with the vapor is
studied.  When a screw dislocation line reaches the boundary of the
material, atoms are adsorbed on the surface in such a way that a
spiral is generated; moreover, under appropriate physical assumptions,
these authors prove that the geometric law governing the dynamics of
the growth of the spiral is precisely given by \eqref{eq:law} where
$-c$ denotes a critical value of the curvature.  We mention that there
is an extensive literature in physics dealing with crystal growth in
spiral patterns. 

\paragraph{Different mathematical approaches.} First and foremost,
defining geometric flows by studying non-linear parabolic equations is
an important topic both in analysis and geometry. Giving references in
such a general framework is out of the scope of this paper. As far as
the motion of spirals is concerned, the study of the dynamics of
spirals have been attracting a lot of attention for more than ten
years.  Different methods have been proposed and developed in order to
define solutions of the geometric law~\eqref{eq:law}. {\color{black} A
  brief list is given here.} A phase-field approach was first proposed
in \cite{kp} and the reader is also referred to
\cite{on00,on03}. Other approaches have been used; for instance,
``self-similar'' spirals are constructed in \cite{ishimura98} by
studying an ordinary differential equation.  In \cite{gik02}, spirals
moving in (compact) annuli with homogeneous Neumann boundary condition
are constructed. From a technical point of view, the classical
parabolic theory is used to construct smooth solutions of the
associated partial differential equation; in particular, gradient
estimates are derived. We point out that in \cite{gik02}, the
geometric law is anisotropic, and is thus more general than
\eqref{eq:law}. In \cite{smereka00,ohtsuka03}, the geometric flow is
studied by using the level-set approach \cite{os88,cgg91,es91}. As in
\cite{gik02}, the author of \cite{ohtsuka03} considers spirals that
typically move into a (compact) annulus and reaches the boundary
perpendicularly.  \medskip

The starting point of this paper is the following fact: to the best of our
knowledge, no geometric flows were constructed to describe the dynamics
of spirals by mean curvature by taking into account both the
singularity of the pinned point and the unboundedness of the domain.

\paragraph{The equation of interest.}
We would like next to explain with more details the main aims of our
study. By parametrizing spirals, we will see (\textit{cf.}
Subsection~\ref{subsec:level-set}) that the geometric
law~\eqref{eq:law} is translated into the quasi-linear parabolic
equation \eqref{eq::4a}. We note that the coefficients are unbounded
(they explode linearly with respect to $r$) and that the equation is
singular: indeed, as $r \to 0$, either $r u_t \to 0$ or first order
terms explode.  Moreover, initial data are also unbounded. In such a
framework, we would like to achieve: uniqueness of weak (viscosity)
solutions for the Cauchy problem for a large class of initial data, to
construct a unique steady state (i.e. a solution of the form $\lambda
t + \bar\varphi(r)$), and finally to show the convergence of general
solutions of the Cauchy problem to the steady state as time goes to
infinity.  This paper is mainly concerned with proving a uniqueness
result and constructing a weak (viscosity) solution; the study of
large time asymptotic will be achieved in \cite{fr2}.

\paragraph{Classical parabolic theory.}
Classical parabolic theory \cite{fried,LSU} could help us to construct
solutions but there are major difficulties to overcome.  For instance,
Giga, Ishimura and Kohsaka \cite{gik02} studied a generalization of
\eqref{eq::4a} in domains of the form $R_{a,b}= {\color{black} \{ a <
  r < b \}}$ with $a>0$ and $b>0$, with Neumann boundary conditions at
$r=a,b$. Roughly speaking, we can say that our goal is to see what
happens when $a \to 0$ and $b \to \infty$. First, we mentioned above
that the equation is not (uniformly) parabolic in the whole domain
$R_{0,\infty}={\color{black} \left\{0<r<+\infty\right\}}$. Second, in
such analysis, the key step is to obtain gradient
estimates. Unfortunately, the estimates from \cite{gik02} in the case
of \eqref{eq::4a} explode as $a$ goes to $0$.  Third, once a solution
is constructed, it is natural to study uniqueness but {\em even in the
  setting of classical solutions} there are substantial difficulties.
To conclude, classical parabolic theory can be useful in order to get
existence results, keeping in mind that getting gradient estimates for
\eqref{eq::4a} is not at all easy, but such techniques will not help
in proving uniqueness.

Recently, several authors studied uniqueness of quasilinear equations
with unbounded coefficients (see for instance \cite{BBBL03,ck}) by
using viscosity solution techniques for instance. But unfortunately,
Eq.~\eqref{eq::4a} does not satisfy the assumptions of these papers.

\paragraph{Main new ideas.}
New ideas are thus necessary to handle these difficulties, both for
existence and uniqueness. As far as uniqueness is concerned, one has
to figure out what is the relevant boundary condition at $r=0$.
We remark that solutions of \eqref{eq::4a} satisfy at least formally a
Neumann boundary condition at the origin
\begin{equation}\label{neum}
0=c  + 2 \bar{u}_r \quad \mbox{for}\quad r=0.
\end{equation}
In some sense, we thus can say that the boundary condition is embedded
into the equation. Second, taking advantage of the fact that the
Neumann condition is compatible with the comparison principle,
viscosity solution techniques (also used in \cite{BBBL03}) permit us
to get uniqueness even if the equation is degenerate and also in a
very large class of weak (sub- and super-) solutions.

But there are remaining difficulties to be overcome.  First, the
Boundary Condition~\eqref{neum} is only true asymptotically (as $r\to
0$) and the fact that it is embedded into the equation makes it
difficult to use. We will overcome this difficulty by making a proper
change of variables (namely $x=\ln r$, see below for further details)
{\color{black} and proving a comparison principle (whose proof is
  rather involved; in particular many new arguments are needed in
  compare with the classical case) in this framework}.  Second,
classical viscosity solution techniques for parabolic equations do not
apply directly to \eqref{eq::4a} because of polar coordinates.  More
precisely, the equation do not satisfy the fundamental structure
conditions as presented in \cite[{Eq.~(3.14)}]{cil92} when polar
coordinates are used.  But the mean curvature equation has been
extensively studied in Cartesian coordinates \cite{es91,cgg91}. Hence
this set of coordinates should be used, at least far from the origin.

\paragraph{Perron's method and smooth solutions.}
We hope we convinced the reader that it is really useful, if not
mandatory, to use viscosity solution techniques to prove uniqueness.
It turns out that it can also be used to construct solutions by using
Perron's method \cite{ishii}. This technique requires to construct
appropriate barriers and we do so for a large class of initial
data. The next step is to prove that these weak solutions are smooth
if additional growth assumptions on derivatives of initial data are
imposed; we get such a result by deriving non-standard gradient
estimates (with viscosity solution techniques too).

We would like also to shed some light on the fact that this notion of
solution is also very useful when studying large time asymptotic (and
more generally to pass to the limit in such non-linear
equations). Indeed, convergence can be proved by using the
half-relaxed limit techniques if one can prove a comparison
principle. See \cite{fr2} for more details.

\subsection{The geometric formulation}
\label{subsec:level-set}

In this section, we make precise the way spirals are 
defined. We will first define them as parametrized curves.

\paragraph{Parametrization of spirals.}
We look for interfaces $\Gamma$ parametrized as follows: 
$\Gamma = \{ r e^{ -i \bar u(r)} : r \ge 0 \} {\color{black} \ \subset \C}$ for some
function $\bar u :[0,+\infty) \to \R$.  If now the spiral
moves, i.e. evolves with a time variable $t >0$, then the function
$\bar u$ also depends on $t>0$.  
\begin{defi}[Spirals]
 A moving \emph{spiral} is a family of curves $(\Gamma_t)_{t>0}$ of
 the following form
\begin{equation}\label{spiral:param}
\Gamma_t = \{ r e^{i\theta} : r >0, \theta \in \R, \; \theta + \bar u (t,r) = 0 \}
\end{equation}
for some function $\bar u : [0,+\infty) \times [0,+\infty) \to \R$.
This curve is oriented by choosing the normal vector field equal to
$(-i + r \partial_r \bar{u} (t,r))e^{-i\bar{u}(t,r)}$.
\end{defi}
\begin{figure}[ht]
\centering\epsfig{figure=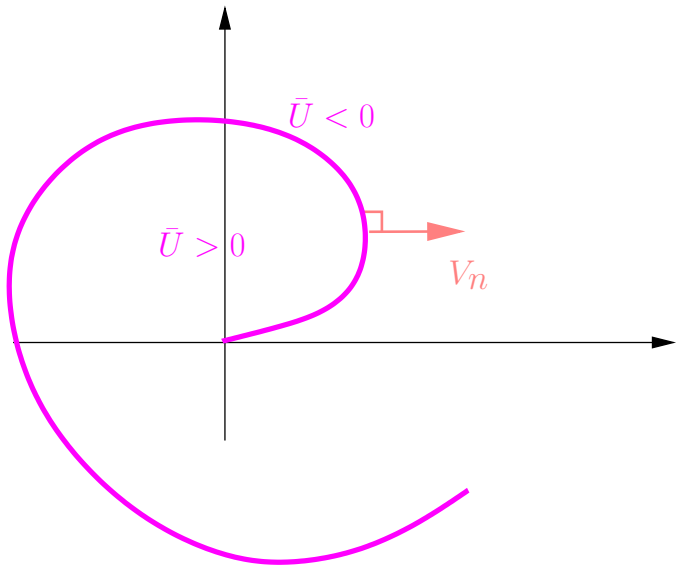,width=90mm}
\caption{Motion of the spiral}\label{f1}
\end{figure}
With the previous definition in hand, the geometric law~\eqref{eq:law}
implies that $\bar u$ satisfies \eqref{eq::4a} with the initial condition
\begin{equation}\label{eq::4b}
\bar{u}(0,r)=\bar{u}_0(r) \quad \mbox{for}\quad r\in (0,+\infty) \, .
\end{equation}

\paragraph{Link with the level-set approach.}

In view of \eqref{spiral:param}, we see that our approach is closely
related to the level-set one. We recall that the level-set approach
was introduced in \cite{os88,es91,cgg91}; in particular, it permits to
construct an interface moving by mean curvature type motion, that is to say
satisfying the geometric law~\eqref{eq:law}. It consists in defining
the interface $\Gamma_t$ as the $0$-level set of a function $\tilde{U}
(t,\cdot)$ and in remarking that the geometric law is verified only if
$\tilde{U}$ satisfies a non-linear evolution equation of parabolic
type. In an informal way, we can say that the quasi-linear evolution
equation~\eqref{eq::4a} is a "graph" equation associated with the
classical mean curvature equation (MCE), but written in polar
coordinates.

More precisely, if ${\color{black} \tilde{U} (t,X)=\theta+ \bar u (t,r)}$ 
with $X {\color{black} =(r\cos\theta,r\sin\theta)}\in \R^2$, then $\bar u$ will satisfy
\eqref{eq::4a} as long as $\tilde{U}$ solves {\color{black} the following level-set equation}
\begin{equation}\label{eq::rm2}
\tilde{U}_t = c |D_X\tilde{U}| 
 + \widehat{D_X\tilde{U}}^\perp \cdot D^2_{XX}\tilde{U} \;  \widehat{D_X\tilde{U}}^\perp
{\color{black} \quad \mbox{for}\quad X\not= 0}
\end{equation}
(where $\hat{p}=p/|p|$ and $p^\perp =(-p_2,p_1)$ for $p=(p_1,p_2)\in\R^2$).
{\color{black} Notice that the angle $\theta$ is multivalued, i.e. only defined modulo $2\pi$.
Such an approach is for instance systematically developed in \cite{ohtsuka03}.}

\subsection{Main results}

\paragraph{Comparison principle.}
{\color{black} Our first main result is a comparison principle}: it says that all
sub-solutions lie below all super-solutions, provided they are ordered
at initial time.
\begin{theo}[Comparison principle for
 \eqref{eq::4a}] \label{thm comp spirales sous lin en r}
Assume that {\color{black} $\bar{u}_0:(0,+\infty)\to \R$} is a {\color{black} globally} Lipschitz continuous function. 
Consider a sub-solution $\bar u$ and a super-solution $\bar v$ of
{\color{black} \eqref{eq::4a},\eqref{eq::4b} (in the sense of Definition \ref{defi::1})} 
such that there exist $C_1>0$ {\color{black} and
for all $t \in [0,T)$} and $r>0$,
{\color{black} \begin{equation}\label{eq:uetv}
\bar u(t,r)  -\bar u_0(r)\le C_1 \quad \text{and} \quad
\bar v(t,r) -\bar u_0(r) \ge -C_1.
\end{equation}}
If $\bar u (0,r) \le \bar u_0 (r) \le \bar v(0,r)$ for all $r \ge 0$, 
then $\bar u \le \bar v$ in {\color{black} $[0,T) \times (0,+\infty)$.}
\end{theo}
\begin{rem}
The growth of the sub-solution $u$ and the super-solution $v$ is made precise
by assuming Condition~\eqref{eq:uetv}. Such a condition is motivated by
the large time asymptotic study  carried out in \cite{fr2}; indeed, we construct
in \cite{fr2} a global solution of the form $\lambda t + \bar u_0(r)$. 
\end{rem}
The proof of Theorem \ref{thm comp spirales sous lin en r} is rather
involved and we will first state and prove a comparison principle in
the set of bounded functions for a larger class of equations (see
Theorem~\ref{thm comp spirales}). We do so in order to exhibit
the structure of the equation that makes the proof work. We then 
turn to the proof of Theorem \ref{thm comp spirales sous lin en r}.

Both proofs are based on the doubling of variable method, which
consists in regularizing the sub- and super-solutions.  Obviously,
this is a difficulty here because one end point of the curve is
attached at the origin and the doubling of variables at the origin is
not well defined. To overcome this difficulty, we work with
logarithmic coordinates $x=\ln r$ for $r$ close to $0$.  But then the
equation becomes
{\color{black} $$u_t =  c e^{-x} \sqrt{1 + u_x^2} + e^{-2x}u_x
+e^{-2x} \frac{u_{xx}}{1 + u_x^2}$$}
We then apply the doubling of variables in the $x$ coordinates. There
is a persistence of the difficulty, because we have now to bound terms
like
$$
A:= ce^{-x}\sqrt{1+u_x^2} - ce^{-y}\sqrt{1+v_y^2}
$$
that can blow up as $x,y \to -\infty$.  We are lucky enough to be able
to show roughly speaking that $A$ can be controlled by {\color{black}
  the doubling of variable of the term $e^{-2x}u_x$ which appears to
  be the main term (in a certain sense) as $x$ goes to $-\infty$.}

In view of the study from \cite{fr2}, $\bar u_0$ has to be chosen
sub-linear in Cartesian coordinates and thus so are the sub- and
super-solutions to be compared. The second difficulty arises when
passing to logarithmic coordinates for large $r$'s; indeed, the
sub-solution and the super-solution then grow exponentially in $x=\ln
r$ at infinity and we did not manage to adapt the previous reasoning
in this setting. There is for instance a similar difficulty when
dealing with the mean curvature equatio. Indeed, in this framework,
for super-linear initial data, the uniqueness of the solution is not
known in full generality (see \cite{BBBL03,ck}).  In other words, the
change of variables do not seem to work far from the origin. We thus
have to stick to Cartesian coordinates for large $r$'s {\color{black}
  (using a level-set formulation)} and see the equation in different
coordinates when $r$ is either small or large (see
Section~\ref{sec:sublinear}).

\paragraph{Existence theorem.}
In order to get an existence theorem, we have to restrict the growth
of derivatives of the initial condition. We make the following
assumptions: the initial condition is globally Lipschitz continuous
and its mean curvature is bounded. We recall that the mean curvature
of a spiral parametrized by $\bar u$ is defined by
$$
\kappa_{\bar u}(r)=u_r\left(\frac {2+(r
   \bar u_r)^2}{(1+(r\bar u_r)^2)^{\frac32}}\right) + \frac{r \bar
 u_{rr}}{(1+(r\bar u_r)^2)^{\frac 32}}.
$$
We can now state our second main result.
\begin{theo}[The general Cauchy problem]\label{th:cauchywc}
 Consider $\bar{u}_0 \in {\color{black} W^{2,\infty}_{loc}(0,+\infty)}$. Assume that $\bar{u}_0$
 is {\color{black} globally Lipschitz continuous and that $\kappa_{\buo}\in L^\infty(0,+\infty)$.}
 Then there exists a unique solution $u$ of
 \eqref{eq::4a},\eqref{eq::4b} {\color{black}on $[0,+\infty)\times (0,+\infty)$ (in the sense of Definition \ref{defi::1})}  
such that for all $T>0$, there exists
 {\color{black} $\bar C_T >0$} such that for all $t\in {\color{black} [0,T)}$ and $r>0$,
\begin{equation}\label{baru:baru0sst}
|\bar u(t,r) -\bar u_0(r)|\le {\color{black} \bar C_T} .
\end{equation}
Moreover, $\bar u$ is Lipschitz continuous with respect to space and
$\frac 12$-H\"older continuous with respect to time. More precisely,
there exists a constant $C$ depending only on $|\buor|_\infty$ and
$|\kappa_{\buo}|_\infty$ such that
$$
|\bar u (t, r+\rho)-\bar u(t,r)|\le C|\rho|
$$
and
\begin{equation}\label{eq::rm1}
|\bar u (t+h,r)-\bar u(t,r)|\le C\sqrt{|h|}.
\end{equation}
\end{theo}
{\color{black} 
\begin{rem}
  Notice that Theorem \ref{th:cauchywc} allows us to consider an
  initial data $\bar u_0$ which does not satisfy the compatibility
  condition (\ref{neum}), like for instance $\bar u_0\equiv 0$ with
  $c=1$. Notice also that we do not know if the solution constructed
  in Theorem \ref{th:cauchywc} is smooth (i.e. belongs to
  $C^\infty((0,+\infty)^2)$).
\end{rem}

To get such a result, we first construct smooth solutions requiring that 
the compatibility condition (\ref{neum}) is satisfied by the
initial datum, like in the following result.}

\begin{theo}[Existence and uniqueness of smooth solutions for the Cauchy problem]\label{th:cauchy}
 Assume that {\color{black} $\bar{u}_0\in W^{2,\infty}_{loc}(0,+\infty)$ with 
$$(\bar{u}_0)_r \in W^{1,\infty}(0,+\infty)\quad \mbox{or}\quad \kappa_{\bar u_0}\in L^\infty(0,+\infty)$$} 
and
 that it satisfies the following compatibility condition for some  $r_0>0$:
\begin{equation}\label{cond:compatibilite}
|c+\kappa_{\buo}|\le C r\quad {\rm for }\quad 0\le r\le {\color{black} r_0}.
\end{equation}
Then there exists a unique continuous function: $\bar u :
[0,+\infty)\times [0,+\infty)$ which is $C^\infty$ in {\color{black} $(0,+\infty)
\times (0,+\infty)$}, which satisfies \eqref{eq::4a},\eqref{eq::4b} 
{\color{black} (in the sense of Definition \ref{defi::1})},
and such that there exists $\bar C >0$ such that
{\color{black} 
$$|\bar u(t,r+\rho) -\bar u(t,r)|\le \bar C |\rho|$$
and
\begin{equation}\label{baru:baru0}
|\bar u(t+h,r) -\bar u(t,r)|\le \bar C |h|.
\end{equation}}
\end{theo}
{\color{black} \begin{rem} Condition \eqref{cond:compatibilite} allows
    us also to improve the H\"{o}lder estimate (\ref{eq::rm1}) and to
    replace it by the Lipschitz estimate (\ref{baru:baru0}).  With the
    help of this Lipschitz estimate (\ref{baru:baru0}), we can
    conclude that the solution constructed in Theorem \ref{th:cauchy}
    is smooth.  Notice also that our space-time Lipschitz estimates on
    the solution allow us to conclude that $\bar u(t,\cdot)$ satisfies
    (\ref{cond:compatibilite}) with the constant $C$ replaced by some
    possible higher constant.  This implies in particular that $\bar
    u(t,\cdot)$ satisfies the compatibility condition (\ref{neum}) for
    all time $t\ge 0$.
\end{rem}}

{\color{black} 
\paragraph{Open questions.}$\mbox{ }$\\
\noindent {\bf A. Weaker conditions on the initial data }\\
It would be interesting to investigate the existence/non-existence and
uniqueness/non-uniqueness of solutions when we allow the initial data
$\bar u_0$ to be less than globally Lipschitz.  For instance what
happens when the initial data describes an infinite spiral close to
the origin $r=0$, with either $\bar u_0(0^+)=+\infty$ or $\bar
u_0(0^+)=-\infty$?  On the other hand, what happens if the growth of
$\bar u_0$ is super-linear
as $r$ goes to $+\infty$?\\
\noindent {\bf B. More general shapes than spiral }\\
One of our main limitation to study only the evolution of spirals in
this paper is that we were not able to prove a comparison principle in
the case of the general level-set equation (\ref{eq::rm2}).  The
difficulty is the fact that the gradient of the level-set function
$\tilde{U}$ may degenerate exactly at the origin where the curve is
attached. The fact that a spiral-like solution is a graph
$\theta=-\bar u(t,r)$ prevents the vanishing of the gradient of
$\tilde{U}$ at the origin $r=0$. If now we consider more general
curves attached at the origin, it would be interesting to study the
existence and uniqueness/non-uniqueness of solutions with general
initial data, like the curves on Figure \ref{F4}.

\begin{figure}[ht]
\centering\epsfig{figure=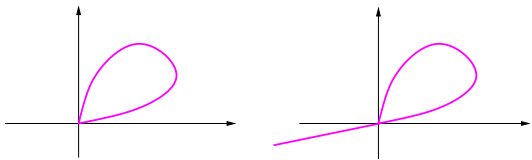,width=70mm}
\caption{Examples of non spiral initial data}\label{F4}
\end{figure}
}

\paragraph{Organization of the article.} In Section~\ref{prel}, we
recall the definition of viscosity solutions for the quasi-linear
evolution equation of interest in this paper.  The change of variables
that will be used in the proof of the comparison principle is also
introduced. In Section~\ref{sec:bounded}, we give the proof of
Theorem~\ref{thm comp spirales sous lin en r} {\color{black} in the case of bounded solutions}. The
proof in the general case is given in Section~\ref{sec:sublinear}. In
Section~\ref{sec:classical}, a classical solution is constructed under
an additional compatibility condition on the initial datum (see
Theorem~\ref{th:cauchy}). First, we construct a viscosity solution by
Perron's method (Subsection~\ref{subsec:perron}); second, we derive
gradient estimates (Subsection~\ref{subsec:gradestim}); third, we
explain how to prove that the viscosity solution is in fact a
classical one (Subsection~\ref{subsec:class}). The construction of the
solution without compatibility assumption (Theorem~\ref{th:cauchywc})
is made in Section~\ref{sec:withoutcompatibility}. Finally, proofs of
technical lemmas are gathered in Appendix~\ref{app:a}.

\paragraph{Notation.} If $a$ is a real number, $a_+$ denotes $\max(0,a)$ and $a_-$  denotes $\max(0,-a)$
If $p=(p_1,p_2) \in \R^2$, $p \neq 0$, then $\hat p$ denotes $p / |p|$
and $p^\perp$ denotes $(-p_2,p_1)$.

\section{Preliminaries}
\label{prel}

\subsection{Viscosity solutions for the main equation}

In view of \eqref{eq::4a}, it is convenient to introduce the
following notation
\begin{equation}\label{eq:barF}
\bar{F}(r,q,Y) 
= c \sqrt{1 + r^2q^2} + q \left(\frac{2+ r^2 q^2}{1+ r^2 q^2}\right)
+ \frac{r Y}{1 + r^2 q^2} . 
\end{equation}
We first recall the notion of viscosity solution for an
equation such as \eqref{eq::4a}. 
\begin{defi}[Viscosity solutions for \eqref{eq::4a},\eqref{eq::4b}] \label{defi::1} ~ \\
{\color{black} Let $T\in (0,+\infty]$.}
A lower semi-continuous (resp. upper semi-continuous) function {\color{black} $u:
 [0,T) \times (0,+\infty) \to \R$} is a (viscosity)
 \emph{super-solution} (resp. \emph{sub-solution}) of
 \eqref{eq::4a},\eqref{eq::4b} {\color{black} on $[0,T)\times (0,+\infty)$} if for any $C^2$ test
 function $\phi$ such that $u-\phi$ reaches a local minimum
 (resp. maximum) at {\color{black} $(t,r) \in [0,T) \times (0,+\infty)$}, we
 have
\begin{itemize}
\item[(i)] If {\color{black} $t> 0$}:
\begin{equation*}
r\phi_t \ge \bar F(r,\phi_r,\phi_{rr})\quad 
\bigg( \text{resp. } r\phi_t \le \bar F(r,\phi_r,\phi_{rr}) \bigg).
\end{equation*}
\item[(ii)] If $t=0$:
$$u(0,r)\ge \bar u_0(r)\quad \bigg({\rm resp.}\; u(0,r)\le \bar u_0(r)\bigg).$$
\end{itemize}
A continuous function {\color{black}  $u: [0,T) \times (0,+\infty) \to \R$} is a
(viscosity) \emph{solution} of \eqref{eq::4a},\eqref{eq::4b} {\color{black} on $[0,T)\times (0,+\infty)$} if it is
both a super-solution and a sub-solution.
\end{defi}
\begin{rem}
We do not impose any condition at $r=0$; in other words,
it is not necessary to impose a condition on the whole parabolic boundary
of the domain. This is due to the ``singularity'' of our equation at $r=0$.
\end{rem}
Since we only deal with this weak notion of solution,
(sub-/super-)solutions will always refer to (sub-/super-)solutions in
the viscosity sense.  \medskip

When constructing solutions by Perron's method, it is necessary to use
the following classical discontinuous stability result. The reader is
referred to \cite{cgg91} for a proof. 
\begin{pro}[Discontinuous stability] \label{pro:stability}
Consider a family $(u_\alpha)_{\alpha \in \mathcal{A}}$ of sub-solutions of
\eqref{eq::4a},\eqref{eq::4b} which is uniformly bounded from above. Then
the upper semi-continuous envelope of $\sup_{\alpha \in \mathcal{A}}
u_\alpha$ is a sub-solution of \eqref{eq::4a},\eqref{eq::4b}. 
\end{pro}

\subsection{A change of unknown function}
\label{subsec:change}

We will make use of the following change of unknown function:
$u(t,x)=\bar{u}(t,r)$ with $x=\ln r$ satisfies for all $t>0$ and $x \in \R$
\begin{equation}\label{eq:polar-spir}
u_t =  c e^{-x} \sqrt{1 + u_x^2} + e^{-2x}u_x
+e^{-2x} \frac{u_{xx}}{1 + u_x^2}
\end{equation}
submitted to the initial condition: for all $x \in \R$,
\begin{equation}\label{eq:ci}
u(0,x)=u_0(x) 
\end{equation}
where $u_0(x) = \bar{u}_0 (e^x)$. 
Eq.~\eqref{eq:polar-spir} can be rewritten $u_t = F(x,u_x,u_{xx})$
with
\begin{equation}\label{defi:F}
F(x,p,X) = c e^{-x} \sqrt{1 + p^2} + e^{-2x}p
+e^{-2x} \frac{X}{1 + p^2} .
\end{equation}
Remark that functions $F$ and $\bar{F}$ are related by the following formula
\begin{equation}\label{eq FF'}
F(x,u_x,u_{xx}) = \frac1r \bar{F} (r, \bar{u}_r, \bar{u}_{rr}) \, .
\end{equation}
Since the function $\ln$ is increasing and maps $(0,+\infty)$ onto
$\R$, we have the following elementary lemma which will be used
repeatedly throughout the paper. 
\begin{lem}[Change of variables]\label{lem:equiv}
 A function $\bar u$ is a solution of \eqref{eq::4a},\eqref{eq::4b}
 if and only if the corresponding function $u$ is a solution of
 \eqref{eq:polar-spir}-\eqref{eq:ci} with $u_0 (x) =
 \bar{u}_0(e^x)$.
\end{lem}
The reader is referred to \cite{cil92} (for instance) for a proof of
such a result.
When proving the comparison principle in the general case, we will also have
to use Cartesian coordinates. From a technical point of view, the following lemma
is needed. 
\begin{lem}[Coming back to the Cartesian coordinates]\label{lem:cartesien}
 We consider a sub-solution $u$ (resp. super-solution $v$) of
 \eqref{eq:polar-spir}-\eqref{eq:ci} and we define the function $\tilde U$
 (resp. $\tilde V$) $: (0,+\infty) \times \R^2 \to \R$ by
$$\tilde 
U(t,X)=\theta(X)+u(t,x(X)) \quad \text{\rm (resp. }
\quad \tilde V(t,Y)=\theta(Y)+u(t,x(Y)) \text{)}  
$$
where $(\theta(Z),x(Z))$ is defined such that
$Z=e^{x(Z)+i\theta(Z)}\ne 0$. 
Then $\tilde U$ (resp. $\tilde V$) is sub-solution (resp. super-solution) of 
\begin{equation}\label{eq:1000}
\left\{\begin{array}{l}
\displaystyle{w_t = c |Dw| + \frac {Dw^\perp}{|D w|} D^2 w \frac {Dw^\perp}{|D w|}}\\
\\
w(0,x)=\theta(X)+\bar u_0(x(X)).
\end{array}\right.
\end{equation}
\end{lem}

\begin{rem}
 In Lemma \ref{lem:cartesien}, for $Z\not=0$, the angle $\theta(Z)$
 is only defined modulo $2\pi$, but is locally uniquely defined by
 continuity. Then $D\theta, D^2\theta$ are always uniquely defined.
\end{rem}

\section{A comparison principle for bounded solutions}\label{sec:bounded}

As explained in the introduction, we first prove a comparison
principle for \eqref{eq::4a} in the class of bounded weak (viscosity)
solutions. In {\color{black} comparison} with classical comparison results for geometric
equations (see for instance \cite{es91,cgg91,sato94,gs91}), the
difficulty is to handle the singularity at the origin ($r=0$).

In order to clarify why  a comparison principle holds true
for such a singular equation, we consider the following generalized
case
\begin{equation}\label{eq:gen-r}
\bar{u}_t = \frac{\bar{b}(\bar{u}_r,r \bar{u}_r)}{r} 
+ \sigma^2 (r \bar{u}_r) \bar{u}_{rr}
\end{equation}
which can be written, with $x = \ln r$, 
\begin{equation}\label{eq:gen-x}
u_t = e^{-x} b(e^{-x} u_x,u_x) + e^{-2x} \sigma^2 (u_x) u_{xx}
\end{equation}
where $b(q,p)= \bar{b} (q,p) - \sigma^2 (p)q$.
\paragraph{Assumption on $\mathbf (b,\sigma)$.}
\begin{itemize}
\item $\sigma \in W^{1,\infty} (\R)$;
\item There exists $\delta_1, \delta_2, \delta _3, \delta_4>0$ such that 
\begin{itemize}
\item
for all $q \in \R$ and $p_1,p_2 \in \R$, 
$$
|b(q,p_1)- b(q,p_2)| \le \delta_1 |p_1 -p_2|;
$$
\item
for all $p \in \R$ and $q_1 \le q_2$, 
$$
\delta_2 (q_2 - q_1) \le b(q_2,p) - b(q_1,p);
$$
$$|b(q_1,p)-b(q_2,p)|\le \delta _3 |q_1-q_2|;
$$
\item for all $p\in \R$
$$|b(0,p)|\le \delta_4 \sqrt {1+|p|^2}$$
\item
{\color{black} we have $\|\sigma\|_\infty^2 < 2 \delta_2$.}
\end{itemize}
\end{itemize}
In our special case, $\sigma (p) = (1+p^2)^{-\frac12}$ and
$b(q,p) = c \sqrt{1+p^2} + q $, {\color{black} and the assumption on $(b,\sigma)$ is satisfied.}
\begin{theo}[Comparison principle for
 \eqref{eq:gen-r}-\eqref{eq:ci}] \label{thm comp spirales}
 Assume that $u_0:(0,+\infty)\to \R$ is Lipschitz continuous.  Consider a bounded
 sub-solution $u$ and a bounded super-solution $v$ of
 \eqref{eq:gen-r},\eqref{eq:ci} in the sense of Definition \ref{defi::1} with $\bar F$ 
given by the right hand side of (\ref{eq:gen-r}). 
Then $u \le v$ in $(0,+\infty) \times
 \R$.
\end{theo}
{\color{black} \begin{rem}
For radial solutions of the heat equation $u_t = \Delta u$ on $\R^n\backslash \left\{0\right\}$, we get $b(q,p)=(n-1)q$ and $\sigma(p)=1$. Therefore
the assumption on $(b,\sigma)$ is satisfied if and only if $1<2(n-2)$. 
Notice that in particular for $n=2$ this assumption is not satisfied.
\end{rem}}
\begin{proof}[Proof of Theorem \ref{thm comp spirales}]
We classically fix $T>0$ and argue by contradiction by assuming that
$$
M = \sup_{0<t<T,x \in \R} (u(t,x) - v(t,x)) >0.
$$ 
\begin{lem}[Penalization] \label{lem:penal}
For $\alpha, \eps, \eta>0$ small enough, and any {\color{black} $K\ge 0$,}
the supremum 
$$
M_{\eps, \alpha} = \sup_{0< t <T,x,y \in \R} \left\{ u(t,x) - v(t,y) -
 e^{Kt} \frac{(x-y)^2}{2 \eps} - \frac{\eta}{T-t} - \alpha
 \frac{x^2}2   \right\}
$$
is attained at $(t,x,y)$ with $t>0$,  $M_{\eps,\alpha} > M/3>0$,
$|x-y| \le C_0 \sqrt{\eps}$ and $\alpha |x| \le C_0 \sqrt{\alpha}$
for some $C_0>0$ only depending on $\|u\|_\infty$ and $\|v\|_\infty$. 
\end{lem}
\begin{proof}[Proof of Lemma~\ref{lem:penal}]
The fact that $M>0$ means that there exist $t^*>0$ and $x^* \in \R$ such that
$$
u(t^*,x^*) -  v(t^*,x^*) \ge M/2> 0 .
$$  
Since $u$ and $v$ are bounded functions, $M_{\eps,\alpha}$ is attained at
a point $(t,x,y)$. By optimality of $(t,x,y)$, we have in particular
\begin{multline*}
u(t,x) - v(t,y) - e^{Kt} \frac{(x-y)^2}{2 \eps}  - \frac{\eta}{T-t} 
- \alpha \frac{x^2}2 \\ 
\ge u(t^*,x^*) - v(t^*,x^*) - \frac{\eta}{T-t^*} - \alpha \frac{(x^*)^2}2 
\ge M/3
\end{multline*}
for $\alpha$ and $\eta$ small enough (only depending on $M$).  In
particular,
$$
\frac{(x-y)^2}{2 \eps} + \alpha \frac{x^2}2 \le \|u\|_\infty + \|v\|_\infty \, .
$$
Hence, there exists a constant $C_0$ only depending on $\|u\|_\infty$ and $\|v\|_\infty$ such that
\begin{equation}\label{estim:penal}
|x-y| \le C_0 \sqrt{\eps} \quad \text{ and } \quad \alpha |x| \le C_0 \sqrt{\alpha} \, .
\end{equation}
Assume  that $t=0$. In this case, we use the fact
that $u_0$ is Lipschitz continuous and \eqref{estim:penal} in order
to get
\begin{eqnarray*}
\frac{M}3 &\le& u_0 (x) - u_0 (y)\le \|D u_0\|_\infty |x-y|\le C_0 \|D
u_0\|_\infty\sqrt\e
\end{eqnarray*}
which is absurd for $\e$ small enough (depending only on $M, C_0$ and $ \|D u_0\|_\infty$).
Hence $t>0$ and the proof of the lemma is complete.
\end{proof}
In the remaining of the proof, $\eps$ is fixed (even if we will choose
it small enough) and $\alpha$ goes to $0$ (even if it is not necessary
to pass to the limit).  In view of the previous discussion, we can
assume that, for $\eps$ small enough, we have $t>0$ for all $\alpha>0$
small enough (independent on $\eps$). We thus can write two viscosity
inequalities. It is then classical to use Jensen-Ishii's Lemma and combine
viscosity inequalities in order to get the following result.
\begin{lem}[{\color{black}Consequence of viscosity inequalities}] \label{lem:ji}
\begin{multline} \label{eq:ji}
\frac{\eta}{(T-t)^2} + K e^{Kt} \frac{(x-y)^2}{2\eps}  \le 
e^{-x} b(e^{-x} (p+\alpha x), p+\alpha x) -  e^{-y} b (e^{-y} p, p) \\
+  e^{-2x} \|\sigma\|_\infty^2 \alpha + \frac{e^{Kt}}{\eps}
(e^{-x} \sigma (p+\alpha x) - e^{-y} \sigma (p))^2 
\end{multline}
where $p =   e^{Kt}\frac{(x-y)}\eps$. 
\end{lem}
\begin{proof}[Proof of Lemma~\ref{lem:ji}]
Jensen-Ishii 's Lemma \cite{cil92} implies that for
all $\gamma_1 >0$, there exist four real numbers $a,b,A,B$ such that
\begin{eqnarray}
 \label{eq:visc1}  a &\le&  e^{-x} b(e^{-x} (p+\alpha x), p+\alpha x) 
 + e^{-2x} \sigma^2 (p+\alpha x) (A+\alpha) \\
 \label{eq:visc2}
 b &\ge&  e^{-y} b (e^{-y} p, p) + e^{-2y} \sigma^2 (p) B
\end{eqnarray}
Moreover $a,b$ satisfy the following inequality
\begin{equation}
\label{eq:timederivative}
a-b \ge  \frac{\eta}{(T-t)^2} + K e^{Kt} \frac{(x-y)^2}{2\eps}  
\end{equation}
(it is in fact an equality)
and for any $\gamma_1>0$ small enough, there exist two real numbers
$A,B$ satisfying the following matrix inequality
$$
\left[\begin{array}{cc} A & 0 \\ 0 & -B \end{array}
\right] \le \frac{e^{Kt}}{\eps} (1+{\gamma_1}) \left[\begin{array}{cc} 1 &
   -1 \\ -1 & 1 \end{array} \right] \, .
$$
This matrix inequality implies
\begin{equation}\label{matrix ineq}
 A \xi_1^2 \le B \xi_2^2 + \frac{e^{Kt}}{\eps}(1+{\gamma_1}) (\xi_1 -\xi_2)^2 
\end{equation}
for all $\xi_1,\xi_2 \in \R$. Use this inequality with $\xi_1 = e^{-x} \sigma (p+\alpha x)$
and $\xi_2 = e^{-y} \sigma (p)$ and let $\gamma_1 \to 0$ in order to get the desired inequality.
\end{proof}
We can next make appear which error terms depend on $\alpha$ and 
which ones depend on $\eps$. For all $\gamma_2>0$, 
\begin{equation} \label{eq:1}
\frac{\eta}{T^2} + K e^{Kt}  \frac{(x-y)^2}{2\eps} \le T_\alpha + T_\eps
\end{equation}
where 
\begin{multline*}
T_\alpha = e^{-x} b(e^{-x} (p+\alpha x), p+\alpha x) -  e^{-x} b (e^{-x} p, p) \\
+  \|\sigma\|_\infty^2 e^{-2x}  \alpha +  (1 + \gamma_2^{-1}) 
\frac{e^{Kt-2x}}{\eps} \|\sigma'\|_\infty^2 (\alpha x)^2
\end{multline*} 
and $T_\eps = T_\eps^1 + T_\eps^2$ with
$$
T_\eps^1 =  e^{-x} b(e^{-x} p, p )- e^{-y} b(e^{-y}p,p) \quad \text{and} \quad
T_\eps^2= (1+\gamma_2)\|\sigma\|^2_\infty \frac{e^{Kt-2x}}{\eps} 
(1 - e^{x-y})^2 .
$$

It remains to estimate $T_\eps$.
\begin{lem}[Estimate for $T_\eps$]\label{lem:estim1}
For all $\gamma_3 >0$, 
\begin{equation} \label{eq:3} T_\eps \le 
  ((C - 2 \delta_2)e^{-2x} + \gamma_3 ) e^{Kt}
 \frac{(x-y)^2}\eps + \gamma_3 \eps
\end{equation}
where {\color{black} $C = \frac{\delta_4^2}{4\gamma_3}  +  (1+\gamma_2)
\|\sigma\|^2_\infty + o_\eps (1)$}.
\end{lem}
\begin{proof}
 Through a Taylor expansion, we obtain
\begin{align*}
T_\eps^1 = & -e^{-y_\theta} b( e^{-y_\theta}p,p)(x-y) - e^{-2y_\theta} \partial_q b (e^{-y_\theta}p,p)p(x-y)  \\
\le & - \eps e^{-Kt} b(e^{-y_\theta}p,p) e^{-y_\theta} p - \delta_2  e^{Kt-2y_\theta}\frac{(x-y)^2}{\eps} 
\end{align*}
where $y_\theta = \theta y + (1-\theta) x$ for some $\theta \in [0,1]$,
{\color{black}  and we have used the fact that for all $q\in \R$,  
$$q (b(q,p)-b(0,p)) \ge \delta_2 q^2$$
We} get for any $\gamma_3 >0$
\begin{align*}
T_\eps^1 \le &   {\color{black} \delta_4} e^{-y_\theta} \frac{|x-y|}{\sqrt{\eps}} \sqrt{\eps}\sqrt{1+p^2} 
- 2\delta_2 e^{Kt-2y_\theta} \frac{(x-y)^2}\eps   \\
\le & \frac{\delta_4^2}{4\gamma_3} (1+o_\eps (1))e^{Kt-2x} \frac{(x-y)^2}\eps +  e^{-Kt} \gamma_3 \eps +  \gamma_3
e^{Kt} \frac{(x-y)^2}\eps \\
& - 2\delta_2  (1+o_\eps (1))e^{Kt-2x} \frac{(x-y)^2}\eps,
\end{align*}
{\color{black}where we have used that $y_\theta = x + o_\eps (1)$. 
Now, since $y = x + O(\sqrt{\eps})$, we also have
$$
T_\eps^2 = (1+\gamma_2)(1+o_\eps (1))
\|\sigma\|^2_\infty \frac{e^{Kt-2x}}{\eps} (x-y)^2
$$}
and we can conclude.
\end{proof}
Combining \eqref{eq:1} and \eqref{eq:3}  finally
yields
$$
\frac{\eta}{T^2} + \left( \frac{K}2 - \gamma_3 + (2 \delta_2-C)
 e^{-2x} \right) e^{Kt} \frac{(x-y)^2}{\eps} \le T_\alpha + \gamma_3\eps.
$$
It suffices now to choose $\gamma_i$, $i=1,2,3$, such that $C \le 2 {\color{black} \delta_2}$ and
then choose $K = 2 \gamma_3$ and $\eps$ small enough to get: $\frac{\eta}{2 T^2} \le T_\alpha$. 

The following lemma
permits to estimate $T_\alpha$. 
\begin{lem}[Estimate for $T_\alpha$]\label{lem:estim2}
For $D>0$ large enough, we have
\begin{equation} \label{eq:2}
\frac{T_\alpha}D \le  e^{-2x} {\color{black} (-\alpha x_-+ \alpha x_+)}
+ e^{-x} |\alpha x|+ e^{-2x}  \alpha +  e^{Kt-2x}\frac{\sqrt{\alpha}}{\eps}|\alpha x|.
\end{equation}
\end{lem}
\begin{proof}
Using assumptions on $b$ immediately yields
$$
T_\alpha \le  e^{-2x} {\color{black} (-\delta_2 \alpha x_-+ \delta_3 \alpha x_+)}
+\delta_1 e^{-x} |\alpha x|+
\|\sigma\|_\infty^2 e^{-2x}  \alpha +  (1 + \gamma_2^{-1}) 
\frac{e^{Kt-2x}\sqrt{\alpha}}{\eps} \|\sigma'\|_\infty^2 C_0 |\alpha x|
$$
where we used that $|\alpha x| \le C_0 \sqrt{\alpha}$. 
\end{proof}
We next consider $a>0$ such that for all $x \le -a$, we have
{\color{black} $$
-|x|e^{-2x} + 2|x|e^{-x} + e^{-2x} \le 0.
$$}
We now distinguish cases. 

{\bf Case 1: $\mathbf{x_n \le -a}$ for some $\mathbf{\alpha_n \to 0}$.} We choose $n$ large enough
so that $e^{Kt} \frac{\sqrt{\alpha_n}}\eps  \le {\color{black} 1}$ and we get
{\color{black} $$
\frac{\eta}{2T^2} \le D \alpha  (-|x|e^{-2x} + 2|x|e^{-x} + e^{-2x}  ) \le 0
$$}
which implies $\eta \le 0$. {\color{black} Contradiction.} 

{\bf Case 2: $\mathbf{x \ge -a}$ for all $\alpha$ small enough.} We use \eqref{estim:penal} and get
$$
\frac{\eta}{{\color{black} 2}T^2} \le  D e^{2a} (2 \alpha |x| + \alpha + e^{Kt} \frac{\sqrt{\alpha}}\eps |\alpha x|)
\le D ( 2 C_0 \sqrt{\alpha} + \alpha + e^{Kt} \frac{C_0 \alpha}{\eps})
$$
and we let $\alpha \to 0$ to get $\eta \le 0$ in this case too.
The proof  is now complete. 
\end{proof}


\section{Comparison principle for sub-linear solutions}
\label{sec:sublinear}

\begin{proof}[Proof of Theorem~\ref{thm comp spirales sous lin en r}]
Thanks to the change of unknown function described in Subsection~\ref{subsec:change},
we can consider the functions $u$ and $v$ defined on $(0,+\infty) \times \R$ which
are sub- and super-solutions of \eqref{eq:polar-spir}. We can either prove that 
$\bar u \le \bar v$ in $(0,+\infty) \times (0,+\infty)$ or that $u \le v$ in 
$(0,+\infty) \times \R$. 

 For $\theta\in \R$, we define
$$
U(t,x,\theta)=\theta+u(t,x)\quad{\rm and}\quad
V(t,x,\theta)=\theta+v(t,x).
$$
Note that $U$ and $V$ are respectively sub and super-solution of
\begin{equation}\label{eq:1010}
W_t(t,x,\theta)=c e^{-x} |D W| + e^{-2x} DW\cdot e_1 + e^{-2x} \frac
{DW^\perp}{|DW|}D^2W \frac {DW^\perp}{|DW|}.
\end{equation}
We fix $T>0$ and we argue by contradiction by assuming that
$$M=\sup_{t \in [0,T], x, \theta\in \R}
\{U(t,x,\theta)-V(t,x,\theta)\}>0.$$  
In order to use the doubling variable technique, we need
a smooth interpolation function $\Psi$ between polar coordinates
for small $r$'s and Cartesian coordinates for large $r$'s. 
Precisely, we choose $\Psi$ as follows.
\begin{lem}[Interpolation between logarithmic and Cartesian
 coordinates]\label{lem:psi}
There exists a smooth ($C^\infty$) function $\psi:\R ^2\mapsto \R^3$
 such that
\begin{equation}
\label{eq:ass-psi}
\left\{\begin{array}{rll}
\psi(x,\theta+2\pi)&=\psi(x,\theta)& \\
\psi(x,\theta)&=(x,e^{i\theta}) &\quad{\rm if} \quad x\le 0\\
\psi(x,\theta)&=(0, e^xe^{i\theta})&\quad {\rm if} \quad x\ge 1
\end{array}\right.
\end{equation}
and such that there exists two constants $\delta_0 >0$ and $m_\psi>0$ such
that
for $x \le 1$ and $\theta \in \R$,
\begin{equation}\label{psi:add}
{\color{black} \mbox{if}\quad \psi(x,\theta) = (a,b) \text{ then }  |b| \le e}
\end{equation}
and such that for all $x,y,\sigma,\theta$, if
$|\psi(x,\theta)-\psi(y,\sigma)| \le \delta_0$ and $|\theta -\sigma | \le
\frac\pi2$, then
\begin{eqnarray}\label{psi:grad-inf}
|\psi(x,\theta)-\psi(y,\sigma)| &\ge& m_\psi
|(x,\theta)-(y,\sigma)|, \\
\label{psi:grad-inf-bis}
\left| D\psi(x,\theta)^T {\color{black} \odot}\ 
(\psi(x,\theta)-\psi(y,\sigma)) \right| &\ge& m_\psi
|(x,\theta)-(y,\sigma)|
\end{eqnarray}
{\color{black} where  $\odot$ is the tensor contraction defined for a $p$-tensor $A=(A_{i_1,\dots, i_p})$ and a
$q$-tensor $B=(B_{j_1,\dots, j_q})$ by}
$$(A\odot B)_{i_1,\dots,i_{p-1},j_2,\dots, j_{q}}=\sum_{k}
A_{i_1,\dots,i_{p-1},k}B_{k,j_2,\dots, j_q}.$$
\end{lem}
The proof of this lemma is given in Appendix~\ref{app:a}. 

\paragraph{Penalization.} We consider the following approximation of $M$
\begin{multline}\label{eq:Mea}
M_{\e,\a}=\sup_{t \in [0,T],x,\theta,y,\sigma \in \R} \bigg\{
U(t,x,\theta)-V(t,y,\sigma) \\ -e^{Kt}\frac{|\psi(x,\theta)-
 \psi(y,\sigma)|^2}{2\e}-\frac 1 \e \left(|\theta-\sigma|-\frac \pi 3\right)_+^2
-\frac \a 2 \left|\psi(x,\theta)\right|^2 -\frac
\eta {T-t}\bigg\}
\end{multline}
where $\e,\a, \eta$ are small parameters and $K\ge
0$ is a large constant to be fixed later. For $\a$ and $\eta$ small
enough we remark that $M_{\e,\a}\ge \frac M 2>0$. In order to prove 
that the maximum $M_{\e,\a}$ is attained, we need the following lemma 
whose proof is postponed until Appendix~\ref{app:a}.
\begin{lem}[A priori estimates]\label{lem:est}
There exists a constant $C_2>0$ such that the following estimate holds
true for any $x,y,\theta,\sigma \in\R$
\begin{eqnarray*}
|u_0(x)-u_0(y)|&\le& C_2 +e^{Kt}\frac
{|\psi(x,\theta)-\psi(y,\sigma)|^2}{4\e}\\
|\theta-\sigma|&\le& C_2+\frac 1 {2\e}
\left(|\theta-\sigma|-\frac \pi 3\right)^2_+.
\end{eqnarray*}
\end{lem}
Using this lemma, we then deduce that
\begin{multline}\label{eq:est1}
U(t,x,\theta)-V(t,y,\sigma) -e^{Kt}\frac 
{|\psi(x,\theta)-\psi(y,\sigma)|^2}{2\e}-\frac 1
\e\left(|\theta-\sigma|-\frac \pi 3\right)^2_+\\
\le  {\color{black} u(t,x)-u_0(x)-v(t,y)
+u_0(y)}+|\theta-\sigma|+|u_0(x)-u_0(y)|\\
-e^{Kt}\frac 
{|\psi(x,\theta)-\psi(y,\sigma)|^2}{2\e}-\frac 1
\e\left(|\theta-\sigma|-\frac \pi 3\right)^2_+\\
\le  2 C_1 + 2 C_2 -e^{Kt}\frac  
{|\psi(x,\theta)-\psi(y,\sigma)|^2}{4\e}-\frac 1
{2\e}\left(|\theta-\sigma|-\frac \pi 3\right)^2_+.
\end{multline}
Using the $2\pi$-periodicity of $\psi$, the maximum is achieved for
$\theta\in [0,2\pi]$. Then, using the previous estimate and the fact
that $-\a(\psi(x,\theta))^2\to -\infty$ as $|x|\to\infty$, we deduce
that the maximum is reached at some point that we still denote
$(t,x,\theta,y,\sigma)$.

\paragraph{Penalization estimates.} 
Using Estimate~\eqref{eq:est1} and the fact that $M_{\e,\a}\ge
0$, we deduce that there exists a constant $C_0=4(C_1+C_2)$ such that
\begin{equation}\label{eq:est2}
\a (\psi(x,\theta))^2+\frac 1
\e\left(|\theta-\sigma|-\frac  \pi3 \right)^2_++e^{Kt}\frac   
{|\psi(x,\theta)-\psi(y,\sigma)|^2}{2\e}\le C_0 .
\end{equation}
On the one hand, an immediate consequence of this estimate is that
$$
|\theta-\sigma|\le \frac  \pi 2
$$
for $\e$ small enough. On the other hand, we deduce from
\eqref{eq:est2} and \eqref{psi:grad-inf}
$$
m_\psi \frac{|\theta-\sigma|^2 + |x-y|^2}{2\e}\le C_0.  
$$
Hence, we  have $|\theta - \sigma| \le \frac{\pi}4$ for $\e$ small enough
so that the constraint $|\theta -\sigma| \le \frac \pi 3$ is not
saturated. We can also choose $\e$ small enough so that 
$$
|x-y| \le \frac12.
$$

In the sequel of the proof, we will also need a better estimate on the
term $\a (\psi(x))^2$; precisely, we need to know that $\a
(\psi(x))^2\to 0$ as $\a\to 0$.  Even if such a result is classical
(see \cite{cil92}), we give details for the reader's convenience.  To
prove this, we introduce
\begin{align*}
 M_{\e,0}= \sup_{t \in [0,T], x,\theta, y, \sigma \in \R} \bigg\{&
 U(t,x,\theta)-V(t,y,\sigma)-e^{Kt}\frac{|\psi(x,\theta)-
   \psi(y,\sigma)|^2}{2\e}\\
 &-\frac 1 \e \left(|\theta-\sigma|-\frac \pi 3\right)^2_+ -\frac \eta
 {T-t}\bigg\}
\end{align*}
which is finite thanks to \eqref{eq:est1}.

We remark that $M_{\e,\a}\le M_{\e,0}$ and that $M_{\e,\a}$ is
non-decreasing when $\alpha$ decreases to zero. We then deduce that
there exists $L$ such that $M_{\e,\a}\to L$ as $\a\to 0$. A simple
computation then gives that
$$
\frac \a 4 (\psi(x,\theta))^2\le M_{\e,\frac \a 2}-
M_{\e,\a}\to 0\quad {\rm as}\quad \a \to 0
$$
and then 
\begin{equation}\label{eq:est3}
\frac \a 2 (\psi(x,\theta))^2\to 0\quad {\rm as}\quad \a \to 0.
\end{equation}

\paragraph{Initial condition.} We now prove the following lemma.
\begin{lem} [Avoiding $t=0$]\label{lem:ci}
For $\e$ small enough, we have $t>0$ for all $\a>0$ small enough.
\end{lem}
\begin{proof}
 We argue by contradiction. Assume that $t=0$. We then 
 distinguish two cases.

If the corresponding $x$ and $y$ are small ($x\le 2$ and
$y\le 2$) then, since $u_0$ is Lipschitz continuous and
\eqref{psi:grad-inf} holds true, there exists a constant $L_0>0$ such
that
\begin{align*}
0<\frac M 2\le M_{\e,\a}\le &
U(0,x,\theta)-V(0,y,\sigma)-\frac{|\psi(x,\theta)-\psi(y,\sigma)|^2}{2\e} \\
\le & L_0|(x,\theta)-(y,\sigma)| - m_\psi\frac
{|(x,\theta)-(y,\sigma)|^2}{2\e} \\
\le & \frac {L_0^2}{2m_\psi} \e
\end{align*}
which is absurd for $\e$ small enough.

The other case corresponds to large $x$ and $y$ ($x\ge 1$ and
$y\ge 1$). In this case, since $\bar u_0$ is Lipschitz continuous,
we know that there exists a constant $L_1>0$
such that
$$
0\le \frac M 2 \le M_{\e, \a}\le |U(0,x,\theta)-V(0,y,
\sigma)|\le |\theta-\sigma| + L_1 |e^{x}-e^{y}|.
$$
Using the fact 
$$
|\theta-\sigma|+ L_1 |e^{x}-e^{y}|\le \left(\frac 1 {m_\psi} +
 L_1\right)  |\psi(x,\theta)-\psi(y,\sigma)|\le \left(\frac 1
 {m_\psi} + L_1\right)\sqrt{ 2 C_0} \sqrt \e
$$
we get a contradiction for $\e$ small enough.
\end{proof}

Thanks to Lemma~\ref{lem:ci}, we will now write two viscosity
inequalities, combine them and exhibit a contradiction. We recall that
we have to distinguish cases in order to determine properly in which
coordinates viscosity inequalities must be written (see the
Introduction).

\paragraph{Case 1:  There exists $\a_n\to 0$ such that $x\ge \frac 32$
 and $y\ge \frac 32$.}
We set $X=e^{x+i\theta}$ and $Y=e^{y+i\sigma}$. Consider $\tilde U$
and $\tilde V$ defined in Lemma~\ref{lem:cartesien}.  Remark that,
even if $\theta(X)$ is defined modulo $2\pi$, the quantity
$\theta(X)-\theta(Y)$ is well defined (for $|X|, |Y|\ge e$ and
$|X-Y|\le \frac 12$) and thus so is $\tilde U(t,X)- \tilde V(t, Y)$.
Recall also that $\tilde{U}, \tilde{V}$ are respectively sub and
super-solutions of the following equation
$$w_t = c|Dw| + \widehat{Dw}^\perp\cdot D^2 w \cdot  \widehat{Dw}^\perp$$

Moreover, using the explicit form of $\psi$, we get that
$$
M_{\e,\a}=\sup_{t \in [0,T], X,Y \in \R^2 \backslash B_1(0)}
\left\{\tilde U(t,X)-\tilde V(t,Y) - \frac 
{e^{Kt}}{2\e} |X-Y|^2 - \frac \a 2 |X|^2 - \frac \eta{T-t}\right\}.
$$

Moreover, $-|D_X \tilde U|\le - \frac 1 {|X|}$ (in the viscosity
sense). We set 
$$p= \frac{X-Y} \e e^{Kt}. 
$$
We now use the
Jensen-Ishii Lemma \cite{cil92} in order to get four real {\color{black} numbers} $a,
b, A, B$ such that
\begin{eqnarray*}
a&\le& c|p+\a X| + \frac {(p+\a X)^\perp}{|p+\a X|}(A+ \a I) \frac
{(p+\a X)^\perp}{|p+\a X|}, \\
b&\ge& c|p| + \frac {p^\perp}{|p|}B \frac
{p^\perp}{|p|}.
\end{eqnarray*}
Moreover, $p$ satisfies the following estimate
\begin{equation}\label{eq:p}
|p+\a X|\ge \frac 1 {|X|},\quad |p|\ge \frac 1 {|Y|}, 
\end{equation}
$a, b$ satisfy the following equality
$$a-b=\frac \eta {(T-t)^2} + K e^{Kt} \frac {|X-Y|^2}{2\e}$$
and $A,B$ satisfy the following matrix inequality
$$\left[\begin{array} {cc}
A&0\\
0&-B
\end{array}\right]
\le \frac {2 e^{Kt}}{\e}
\left[\begin{array} {cc}
I&-I\\
-I&I
\end{array}\right].
$$
This matrix inequality implies 
\begin{equation}\label{eq:matrix}
A\xi_1^2\le B\xi_2^2 + \frac {2 e^{Kt}}{\e} |\xi_1-\xi_2|^2
\end{equation}
for all $\xi_1,\xi_2\in \R^2$. Subtracting the two viscosity
inequalities, we then get
\begin{align*}
\frac \eta {T^2}\le &
c|p+\a X|-c|p| +\a +  \frac {(p+\a X)^\perp}{|p+\a X|}A\frac
{(p+\a X)^\perp}{|p+\a X|}- \frac {p^\perp}{|p|}B \frac
{p^\perp}{|p|}\\
\le& \a |c| |X| +\a+  \frac {2 e^{Kt}}{\e} \left( \frac {p+\a X}{|p+\a
   X|}-\frac {p}{|p|}\right)^2\\
\le & |c| \sqrt{C_0}\sqrt \a +\a+\frac {2 e^{Kt}}{\e}\left(2\left(\frac
   {\a X}{\frac 1 {|X|}}\right)^2 + 2 \left(\frac {p}{|p|} \frac {{\color{black} |\a
     X|}}{|p+|\a X||}\right)^2\right)\\
\le & |c| \sqrt{C_0}\sqrt \a +\a+\frac {8 e^{Kt}}{\e}\left({\a
   |X|^2}\right)^2 
\end{align*}
where we have used successively \eqref{eq:matrix}, \eqref{eq:est2} and
\eqref{eq:p}.  Recalling, by \eqref{eq:est3} that $\a |X|^2=o_\a(1)$,
we get a contradiction for $\a$ small enough.

\paragraph{Case 2: There exists $\a_n\to 0$ such that $x\le -\frac 12$ and
 $y\le -\frac 12$.}
Using the explicit
form of $\psi$ and the fact that $U(t,x,\theta)=\theta+u(t,x)$ and
$V(t,y,\sigma)=\sigma+v(t,y)$ with $u$ and $v$ respectively sub and
super-solution of \eqref{eq:polar-spir}, we remark that
$$
M_{\e,\a}=\sup_{t',x',y'}\{u(t',x')-v(t',y')-e^{Kt'}\frac{|\psi(x',\theta)-
 \psi(y',\sigma)|^2}{2\e}-\frac \a 2 |x'|^2 -\frac \eta {T-t'}
+\theta-\sigma -\frac \a 2\}.
$$ 
Moreover, the maximum is reached at $(t,x,y)$, where we recall that
$(t,x,\theta,y,\sigma)$ is the point of maximum in \eqref{eq:Mea}.
Using the Jensen-Ishii Lemma \cite{cil92}, we then deduce the
existence, for all $\gamma_1 >0$, of four real numbers $a,b,A,B$ such
that
\begin{eqnarray*}
 a &\le&  ce^{-x} \sqrt{1 + (p+\alpha x)^2} + e^{-2x} (p+\alpha x ) 
 + e^{-2x} \frac{A + \alpha }{1 + (p+\alpha x )^2} \\ 
 b &\ge&  ce^{-y} \sqrt{1 + p^2} + e^{-2y} p + e^{-2y} \frac{B}{1+p^2}
\end{eqnarray*}
where 
$$p=\frac{x-y}\e e^{Kt}.$$
These inequalities are exactly \eqref{eq:visc1} and \eqref{eq:visc2}.
Moreover $a,b$ satisfy the following inequality
$$
a-b = \frac{\eta}{(T-t)^2} + K e^{Kt}
\frac{|\psi(x,\theta)-\psi(y,\sigma)|^2}{2\eps}  \ge
\frac{\eta}{(T-t)^2} + K e^{Kt} 
\frac{|x-y|^2}{2\eps}
$$
and we obtain \eqref{eq:timederivative}. Moreover, $A,B$ satisfy the
following matrix inequality
$$
\left[\begin{array}{cc} A & 0 \\ 0 & -B \end{array} \right] \le
\frac{e^{Kt}}{\eps} (1+{\gamma_1}) \left[\begin{array}{cc} 1 & -1 \\
   -1 & 1 \end{array} \right] 
$$
which implies \eqref{matrix ineq}. On one hand, from \eqref{eq:visc1}, \eqref{eq:visc2},
\eqref{eq:timederivative} and \eqref{matrix ineq}, 
we can derive \eqref{eq:ji}.  On the other hand, \eqref{eq:est2}, the
fact that $x\le 0$, $y\le 0$ and Lemma~\ref{lem:psi} imply
\eqref{estim:penal} (with a different constant). We thus can apply Lemmas
\ref{lem:estim1}, \ref{lem:estim2} and deduce the desired
contradiction.

\paragraph{Case 3: There exists $\a_n\to 0$ such that $-1\le x, y \le 2$.}
Since $\psi\in C^\infty $, there then exists $M_\psi>0$ (only
depending on the function $\psi$) such that for all $x \in [-1,2]$ and
$\theta \in [-\pi,3 \pi]$,
\begin{equation}\label{eq:mpsi}
|\psi(x,\theta)| + |D \psi(x,\theta)|+|D^2 \psi(x,\theta)|
+|D^3\psi(x,\theta)| \le M_\psi.
\end{equation}
For simplicity of notation, we denote $(x,\theta)$ by $\bar x$ and
$(y,\sigma)$ by $\bar y$. We next define
$$
p_{\bar x}=\frac {e^{K t }}\e D\psi(\bar x)^T \odot
(\psi(\bar x)-\psi(\bar y))\quad{\rm and}\quad p_{\bar y}=\frac {e^{K
   t }}\e D \psi(\bar y)^T \odot (\psi(\bar x)-\psi(\bar y)).
$$
We have $p_{\bar{x}}, p_{\bar{y}} \in\R^2$ and we set $(e_1,e_2)$ a
basis of $\R^2$.
\begin{lem}[Combining viscosity inequalities for $\alpha=0$]
We have for $\alpha=0$
\begin{equation}\label{eq:301}
 \frac \eta{T^2}+ Km_\psi^2 e^{Kt}\frac {|\bar x-\bar y|^2}{2\e}\le  ce^{-x}
 |p_{\bar x}|- ce^{-y} |p_{\bar y}| + e^{-2x} p_{\bar x}\cdot  e_1 -
 e^{-2y}  p_{\bar y}\cdot e_1 + \frac {2e^{Kt}}\e (\I_1+\I_2)
\end{equation}
where 
\begin{eqnarray*}
 \I_1&=&(\psi(\bar x)-\psi(\bar y))\odot \bigg( D^2\psi(\bar x) e^{-x} \widehat
 {p_{\bar x}^\perp}\cdot e^{-x} \widehat {p_{\bar x}^\perp}
 -D^2\psi(\bar y) e^{-y} \widehat{p_{\bar y}^\perp}\cdot e^{-y} \widehat {p_y^\perp} \bigg) \\
 \I_2&=&\left|D\psi(\bar x)e^{-x} \widehat {p_{\bar x}^\perp} -D\psi(\bar y)e^{-y}
   \widehat {p_{\bar y}^\perp} \right|^2.
\end{eqnarray*}
\end{lem}
\begin{proof} 
Recall that $U$ and $V$ are respectively sub and
super-solution of \eqref{eq:1010} and use the Jensen-Ishii Lemma \cite{cil92} in order to deduce that there exist
two real numbers $a,b$ and two $2 \times 2$ real matrices $A,B$ such that
\begin{eqnarray*}
 a &\le&  {\color{black} c}e^{-x} |\tilde p_{\bar x}| + e^{-2x} \tilde p_{\bar x}\cdot e_1 \\
&&+ e^{-2 x} \frac {\tilde p_{\bar x}^\perp}{|\tilde p_{\bar x}|} \left(A +
\a(\psi(\bar x)\odot D^2\psi(\bar x)+ D\psi(\bar x)^T \odot D\psi(\bar x)
)\right)\frac {\tilde p_{\bar x}^\perp}{|\tilde p_{\bar x}|} \\  
 b &\ge&  {\color{black} c}e^{-y} |p_{\bar y}| + e^{-2y}   p_{\bar y}\cdot e_1 +
 e^{-2y}  \frac {p_{\bar y}^\perp}{|p_{\bar y}|} B \frac {p_{\bar
     y}^\perp}{| p_{\bar y}|}  
\end{eqnarray*}
where 
$$\tilde p_{\bar x}=p_{\bar x}+\a D\psi(\bar x)^T\odot
\psi(\bar x).
$$
Remark that , since $D_\theta U=D_\theta V=1$, there exists
$\delta_0>0$ such that
$$\tilde p_{\bar x}\ge \delta_0>0\quad{\rm and}\quad p_{\bar y}\ge \delta_0>0.$$
Moreover $a,b$ satisfy the following equality
$$
a-b = \frac{\eta}{(T-t)^2} + K e^{Kt}
\frac{|\psi(\bar x)-\psi(\bar y)|^2}{2\eps}  
$$
and $A,B$ satisfy the following matrix inequality 
\begin{multline*}
 \left[\begin{array}{cc} A & 0 \\ 0 & -B \end{array} \right] \le
 \frac{2 e^{Kt}}{\eps}\bigg\{ \left[\begin{array}{cc} (\psi(\bar
     x)-\psi(\bar y))\odot D^2\psi(\bar x) & 0 \\ 0 & -(\psi(\bar
     x)-\psi(\bar y))\odot D^2\psi(\bar
     y) \end{array} \right]\\
 +\left[\begin{array}{cc} D\psi(\bar x)^T \odot D\psi(\bar x) &
     -D\psi(\bar y)^T\odot D\psi(\bar x) \\ -D\psi(\bar y)^T\odot
     D\psi(\bar x) & D\psi(\bar y)^T\odot D\psi(\bar y) \end{array}
 \right]\bigg\} \, .
\end{multline*}
This implies
\begin{multline}\label{eq:ineqmatrix}
A\xi\cdot \xi\le B\zeta\cdot \zeta+ \frac {2e^{Kt}}\e
\bigg\{(\psi(\bar x)-\psi(\bar y))\odot D^2\psi(\bar x) \xi\cdot \xi-
(\psi(\bar x)-\psi(\bar y))\odot D^2\psi(\bar y) \zeta\cdot \zeta\nonumber\\ 
+\left|D\psi(\bar x)\xi-D\psi(\bar y)\zeta\right|^2\bigg\}
\end{multline}
for all $\xi,\zeta\in \R^2$. Combining the two viscosity inequalities
and using the fact that $|\psi(\bar x)-\psi(\bar y)|\ge m_\psi|\bar{x}-\bar{y}|$,
we obtain
\begin{align*}
 \frac \eta{T^2}+ Km_\psi^2 e^{Kt}\frac {|\bar x-\bar y|^2}{2\e}\le & {\color{black} c}e^{-x} |\tilde
 p_{\bar x}|- {\color{black} c}e^{-y} |p_{\bar y}| + e^{-2x} \tilde p_{\bar x}\cdot
 e_1 - e^{-2y} p_{\bar y}\cdot e_1 \\
 +& \a e^{-2 x} . \widehat {\tilde p_{\bar x}^\perp}
\left(\psi(\bar x)\odot D^2\psi(\bar x)+ D\psi(\bar x)^T
   D\psi(\bar x) \right)\widehat  {\tilde p_{\bar x}^\perp} \\
 +& \frac {2e^{Kt}}\e (\tilde \I_1 + \tilde \I_2) 
\end{align*}
where $\tilde \I_1$ and $\tilde \I_2$ are defined respectively as
$\I_1$ and $\I_2$ with $p_{\bar x}$ replaced by $\tilde p_{\bar x}$. 
Remarking that there exists a constant $C>0$ such that 
\begin{eqnarray*} {\color{black} c}e^{-x} |\tilde p_{\bar x}| &+& e^{-2x} \tilde p_{\bar
   x}\cdot e_1+\left|\a e^{-2 x} \frac {\tilde p_{\bar
       x}^\perp}{|\tilde p_{\bar x}|}\left(\psi(\bar x)\odot
     D^2\psi(\bar x)+ D\psi(\bar x)^T D\psi(\bar x) \right)\frac
   {\tilde p_{\bar x}^\perp}{|\tilde p_{\bar x}|}\right|\\
&&  \le  {\color{black} c}e^{-x} | p_{\bar x}| + e^{-2x} p_{\bar
   x}\cdot e_1+C\a \left( |D^2\psi(\bar x)|^2 + |D\psi(\bar x)|^2 + |\psi(\bar
 x)|^2\right)\\
&&  \le {\color{black} c}e^{-x} | p_{\bar x}| + e^{-2x}  p_{\bar
   x}\cdot e_1+3M_\psi^2 C \a
\end{eqnarray*}
and 
\begin{align*}
|\tilde \I_1-\I_1|+|\tilde \I_2 -\I_2|\le& C \left|\widehat {\tilde
   p_{\bar x}^\perp}-\widehat {p_{\bar x}^\perp}\right|\\
\le &C \left|\frac {\tilde p_{\bar x}-p_{\bar x}}{|\tilde p_{\bar x}|}\right|+|p_{\bar x}|\left|\frac
   1{|\tilde p_{\bar x}|}-\frac 1{|p_{\bar x}|}\right|\\
\le &C  \left|\frac {\tilde p_{\bar x}-p_{\bar x}}{\delta_0}\right|+\left|\frac
   {|p_{\bar x}|-|\tilde p_{\bar x}|}{\delta_0}\right|\\
\le &2C \left|\frac {\tilde p_{\bar x}-p_{\bar x}}{\delta_0}\right|\\
\le&\frac{2C^2\a}{\delta_0}
\end{align*}
and sending $\a\to 0$ (recall that $\bar{x},\bar{y}$ lie in a compact domain), we get \eqref{eq:301}. 
\end{proof}
\begin{lem}[Estimate on ${\mathcal I}_1$]\label{lem:tech1}
There exists a constant $\overline C_1$ such that
\begin{equation}\label{estim:i1}
|\I_1|\le \overline C_1 |x-y|^2 \end{equation}
\end{lem}
\begin{proof}
In order to prove \eqref{estim:i1}, we write
\begin{align*}
  \frac{|\I_1|}{|\psi(\bar x)-\psi(\bar y))|} \le &
  |(D^2\psi(\bar x)- D^2\psi(\bar y)) e^{-x} \widehat {p_{\bar x}^\perp}\cdot e^{-x} \widehat   {p_{\bar x}^\perp}|\\
  +&| D^2\psi(\bar y)
  (e^{-x}-e^{-y}) \widehat {p_{\bar x}^\perp}\cdot e^{-x} \widehat{p_{\bar x}^\perp}|\\
  +& |D^2\psi(\bar y) e^{-y} \left(\widehat {p_{\bar
        x}^\perp}-\widehat{p_y^\perp}\right)\cdot e^{-x} \widehat
  {p_{\bar x}^\perp}|\\
  +& |D^2\psi(\bar y)
  e^{-y} \widehat{p_y^\perp} \cdot (e^{-x}-e^{-y}) \widehat {p_{\bar x}^\perp}\\
  +& |D^2\psi(\bar y) e^{-y} \widehat{p_{\bar y}^\perp}\cdot e^{-y}
  \left(\widehat {p_{\bar x}^\perp}-\widehat{p_y^\perp}\right)|.
\end{align*}
Thanks to \eqref{eq:mpsi} and $\max(|x|,|y|) \le 2$, we have
\begin{eqnarray*}
|D^2\psi(\bar x)- D^2\psi(\bar y)| &\le& M_\psi |\bar x-\bar y|, \\
|e^{-x}-e^{-y}| &\le&   e^2 |\bar x-\bar y|.
\end{eqnarray*}
We also have the following important estimate
\begin{align*}
 \left|\widehat {p_{\bar x}^\perp}-\widehat {p_{\bar y}^\perp}\right|\le&
\left|\frac {p_{\bar x}-p_{\bar y}}{|p_{\bar x}|}\right|+|p_{\bar y}|\left|\frac
   1{|p_{\bar x}|}-\frac 1{|p_{\bar y}|}\right|\\
\le &  \left|\frac {p_{\bar x}-p_{\bar y}}{|p_{\bar x}|}\right|+\left|\frac
   {|p_{\bar y}|-|p_{\bar x}|}{|p_{\bar x}|}\right|\\
\le & 2 \left|\frac {p_{\bar x}-p_{\bar y}}{|p_{\bar x}|}\right|\\
\le& 2\frac{\frac {e^{Kt}}\e|D\psi(\bar x)-D\psi(\bar y)||\psi(\bar
 x)-\psi(\bar y)|}{\frac {e^{Kt}}\e m_\psi |\bar x-\bar y|}\\
\le & \frac{2 M_{\psi}^2}{m_\psi}|\bar x-\bar y|
\end{align*}
where we have used the fact that $|p_{\bar x}|\ge \frac {e^{Kt}}\e
m_\psi|\bar x -\bar y|$ (see \eqref{psi:grad-inf-bis}). This finally gives
that there exists a constant $\overline C_1$ (depending on $m_\psi$
and $M_\psi$) such that \eqref{estim:i1} holds true. 
\end{proof}
Using the fact that $|p_{\bar x}|,|p_{\bar y}|\le C \frac {e^{Kt}}\e |\bar x-\bar
y|$, we can prove in a similar way the following lemma.
\begin{lem}[Remaining estimates]\label{lem:tech2}
There exist three positive constants $\overline C_2,\overline C_3$ and $\overline C_4$  such that 
\begin{eqnarray*}
|\I_2| &\le & \overline C_2 |\bar x-\bar y|^2, \\
ce^{-x} |p_{\bar x}|- ce^{-y} |p_{\bar y}|&\le &\overline C_3 \frac
{e^{Kt}}\e |\bar x-\bar y|^2 ,\\
e^{-2x} p_{\bar x}\cdot  e_1 - e^{-2y}  p_{\bar y}\cdot e_1 &\le&
\overline C_4 \frac {e^{Kt}}\e |\bar x-\bar y|^2.
\end{eqnarray*}
\end{lem}

Use now Lemmas~\ref{lem:tech1} and \ref{lem:tech2} in 
order to derive from \eqref{eq:301} the following inequality
$$
\frac \eta{T^2}+ Km_\psi e^{Kt}\frac {|\bar x-\bar
 y|^2}{2\e}\le\overline C \frac {e^{Kt}}\e |\bar x-\bar y|^2
$$
with $\overline C=\overline C_1+\overline C_2+\overline C_3+\overline C_4$.
Choosing $K\ge \frac {2 \overline C}{m_\psi}$, we get a contradiction.
\end{proof} 


\section{Construction of a classical solution}
\label{sec:classical}

{\color{black} In this section, our main goal is to prove Theorem \ref{th:cauchy}
which claims the existence and uniqueness of classical solutions 
under suitable assumptions on the initial data $\bar u_0$.
Notice that assumptions \eqref{cond:compatibilite}  on the initial data
imply in particular that
\begin{equation}\label{cond:compatibilite1}
c+2(\bar{u}_0)_r(0)=0 .\end{equation}}

To prove Theorem \ref{th:cauchy}, we first construct a unique weak (viscosity)
solution. We then prove gradient estimates from which it is not
difficult to derive that the weak (viscosity) solution is in fact
smooth; in particular, it thus satisfies the equation in a classical
sense. 

\subsection{Barriers and Perron's method}\label{subsec:perron}

Before constructing solutions of \eqref{eq::4a} submitted to the
initial condition~\eqref{eq::4b}, we first construct appropriate
barrier functions.

\begin{pro}[Barriers for the Cauchy problem] \label{prop:barriers}
 Assume that {\color{black} $\bar{u}_0\in W^{2,\infty}_{loc}(0,+\infty)$ and
$$(\bar{u}_0)_r \in W^{1,\infty} (0,+\infty)\quad \mbox{or}\quad \kappa_{\bar u_0}\in L^\infty(0,+\infty)$$}
with $\bar{u}_0$
 such that \eqref{cond:compatibilite} holds true.  {\color{black} Then} there exists a constant
 $\bar{C}>0$ such that $\bar{u}^\pm (t,r)=\bar{u}_0 (r) \pm \bar{C}t$ are
 respectively a super- and a sub-solution of
 \eqref{eq::4a},\eqref{eq::4b}.
\end{pro}
\begin{proof}
It is enough to prove that the following quantity is finite
$$
\bar{C} = \sup_{r \ge 0} \frac1r \left| \bar{F}(r,(\bar{u}_0)_r (r), (\bar{u}_0)_{rr}(r))
\right|  = \max (\bar{C}_1,\bar{C}_2) 
$$
with
$$
\bar{C}_1 = \sup_{r \in [0,{\color{black} r_0}]}  \frac{| \bar{F}(r,(\bar{u}_0)_r (r),
 (\bar{u}_0)_{rr}(r))|}r, \quad
\bar{C}_2 = \sup_{r \in [{\color{black} r_0},+\infty)}  \frac{| \bar{F}(r,(\bar{u}_0)_r (r) (x),
 (\bar{u}_0)_{rr}(r))|}r.
$$
On one hand, thanks to \eqref{cond:compatibilite} and the Lipschitz
regularity of $\bar u_0$, we have $\bar{C}_1$ is finite. On the other
hand, thanks to Lipschitz regularity and $(\bar u_0)_r \in
W^{1,\infty}$ {\color{black} or $\kappa_{\bar u_0}\in L^\infty$}, $\bar{C}_2$ is also finite. The proof is now complete.
\end{proof}
We now construct a viscosity solution for
\eqref{eq::4a},\eqref{eq::4b}; this is very classical with the results
we have in hand, namely the strong comparison principle and the
existence of barriers. However, we give a precise statement and a
sketch of proof for the sake of completeness.
\begin{pro}[Existence by Perron's method]\label{pro:cauchy}
 Assume that $\buo\in C(0,+\infty)$ and that there
 exists 
{\color{black} $$u^+(t,r):=\buo(r)+f(t) \quad \left(\mbox{resp.}\ u^-(t,r):=\buo(r)-f(t)\right)$$ 
for some continuous function $f$ satisfying $f(0)=0$, which are respectively a
 super- and a sub-solution of \eqref{eq::4a},\eqref{eq::4b}.} Then,
 there exists a {\color{black} (continuous)} viscosity solution $\bar{u}$ of
 \eqref{eq::4a},\eqref{eq::4b} such that {\color{black} \eqref{baru:baru0sst}} holds true for
some constant {\color{black} $\bar C_T$} depending on {\color{black} $f$}.
 {\color{black} Moreover $\bar u$ is the unique  viscosity solution of
 \eqref{eq::4a},\eqref{eq::4b} such that \eqref{baru:baru0sst} holds
 true.}
\end{pro}
\begin{proof}
In view of Lemma~\ref{lem:equiv}, it is enough to construct a
solution $u$ of \eqref{eq:polar-spir} satisfying \eqref{eq:ci} with
$u_0(x) = \bar{u}_0 ({\color{black} e^x})$. 

Consider the set 
$$
\mathcal{S} = \{ v : (0,+\infty) \times \R \to \R, \text{ sub-solution
 of } \eqref{eq:polar-spir} \text{ s.t. } v \le u^+ \} \, .
$$
Remark that it is not empty since $u^- \in \mathcal{S}$ (where
$u^\pm(t,x)=\bar u^\pm(t,r)$ with $x=\ln r$).  We now
consider the upper envelope $u$ of $(t,r)\mapsto \sup_{v\in {\mathcal
S}} v(t,r)$. By Proposition~\ref{pro:stability}, it is a
sub-solution of \eqref{eq:polar-spir}. The following lemma derives
from the general theory of viscosity solutions as presented in
\cite{cil92} for instance.
\begin{lem}\label{lem:bump}
 The lower envelope $u_*$ of $u$ is a super-solution of
 \eqref{eq:polar-spir}.
\end{lem}
We recall that the proof of this lemma proceeds by contradiction
and consists in constructing a so-called bump function around the
point the function $u_*$ is not a super-solution of the equation. 
The contradiction comes from the maximality of $u$ in $\mathcal{S}$. 

Since for all $v \in \mathcal{S}$,
$$
u_0 (x) - f(t) \le v \le u_0 (x) +f(t) ,
$$
with $f(0)=0$
we conclude that 
$$
u_0(x) = u_* (0,x) = u (0,x) \, .
$$
If $\bar u$ satisfies \eqref{baru:baru0sst}, we use the comparison principle and get $u \le u_*$ in
$(0,T) \times \R$ for all $T>0$. Since $u_* \le u$ by construction, we deduce
that $u=u_*$ is a solution of \eqref{eq:polar-spir}. 
The comparison principle also ensures that the solution we constructed
is unique. The proof of Proposition~\ref{pro:cauchy} is now complete. 
\end{proof}

\subsection{Gradient estimates} \label{subsec:gradestim}

In this subsection, we derive gradient estimates for a viscosity
solution $\bar{u}$ of \eqref{eq::4a} satisfying \eqref{baru:baru0sst}.
\begin{pro}[Lipschitz estimates] \label{prop:gradestim} Consider a {\color{black} globally} Lipschitz continuous
function $\bar{u}_0$. We denote by $L_0>0$ and $L_1 >0$ such that
for all $r>0$,
$$
-L_0 \le (\bar{u}_0)_r (r) \le L_1.
$$
Let $\bar u$ be a viscosity solution $\bar{u}$ of
 \eqref{eq::4a},\eqref{eq::4b} satisfying \eqref{baru:baru0sst}. Then $\bar
 u$ is also Lipschitz continuous {\color{black} in space}: $\forall t>0$, $\forall r \ge
 0$,
{\color{black} \begin{equation}\label{estim:gradient}
\left\{\begin{array}{ll}
-\max(1,L_0) \le \bar{u}_r (t,r)\le L_1 & \quad \mbox{if}\quad c\ge 0\\
\\
-L_0 \le \bar{u}_r (t,r)\le \max(1,L_1) & \quad \mbox{if}\quad c\le 0
\end{array}\right. 
\end{equation}}
Moreover, if {\color{black} $\bar{u}_0\in W^{2,\infty}_{loc}(0,+\infty)$ with
$$(\bar{u}_0)_r\in W^{1,\infty}(0,+\infty)\quad \mbox{and}\quad \kappa_{\bar u_0}\in L^\infty(0,+\infty)$$}
and
\eqref{cond:compatibilite} holds true, then {\color{black} $\bar u$} is $\bar{C}$-Lipschitz
continuous with respect to $t$ for all $r >0$ where $\bar{C}$ denotes
the constant appearing in Proposition~\ref{prop:barriers}.
\end{pro}
\begin{proof} ~ \\
{\color{black} \noindent {\bf Step 1: gradient estimates}}\\
Proving \eqref{estim:gradient} {\color{black} for $c\ge 0$} is equivalent to prove that the
solution $u$ of \eqref{eq:polar-spir} satisfies the following gradient
estimate: $\forall t>0$, $\forall x \in \R$,
\begin{equation}\label{estim:gradient-x}
-\bar L_0e^x \le u_x (t,x)\le L_1 e^x  
\end{equation}
where $\bar L_0=\max(1,L_0)$.
We will prove each inequality separately.
Since $\bar{u}$ is sublinear,
there exists $C_u>0$ such that for all $x \in \R$
$$
|u(t,x)| \le C_u (1 + e^x) .
$$
Eq.~\eqref{estim:gradient} is equivalent to prove
\begin{align*}
{\color{black} M^0} &= \sup_{t \in (0,T), x \le y \in \R} \left\{ u (t,x)+\bar L_0 e^x
 - u (t,y) - \bar L_0 e^y \right\} \le 0 \\
{\color{black} M^1} &= \sup_{t \in (0,T), x \ge y \in \R} \left\{ u (t,x)-L_1 e^x
 - u (t,y) + L_1 e^y \right\} \le 0. 
\end{align*}

We first prove that ${\color{black} M^0} \le 0$. We argue by contradiction by assuming
that ${\color{black} M^0}>0$ and we exhibit a contradiction. The following supremum
$$
{\color{black} M_\a^0}= \sup_{t \in (0,T), x \le y \in \R} \left\{ u(t,x) - u(t,y) +
 \bar L_0 e^x-\bar L_0 e^y -\frac
\a 2 x^2 - \frac \a 2 y^2
- \frac\eta{T-t} \right\}
$$
is also positive for $\a$ and $\eta$ small enough.

Using the fact that, by assumption on $\bar u_0$,
\begin{equation}\label{eq:borne1}
u(t,x)-u(t,y)+\bar L_0e^x-\bar L_0e^y\le
u(t,x)- u_0(x)+ u_0(x)- u_0(y)+ u_0(y)-u(t,y)+\bar L_0 
e^x-\bar L_0 e^y\le 2C_1
\end{equation}
and the fact
that $-\frac \a 2 x^2-\frac \a 2 y^2\to -\infty $ as $x\to \pm \infty$
or $y\to \pm \infty$, we deduce that the supremum is achieved at a point
$(t,x,y)$ such that $t \in (0,T)$ and $x > y$.

Moreover, we deduce using \eqref{eq:borne1} and the fact that $M_\a>0$,
that there exists a constant $C_0:=4C_1$ such
that $x$ and $y$ satisfy the following inequality
$$
\a x^2+ \a y^2\le C_0.
$$

Thanks to Jensen-Ishii's Lemma (see e.g. \cite{cil92}), we conclude that
there exist $a,b,X,Y \in \R$ such that
\begin{eqnarray*}
 a &\le& c e^{-x} \sqrt{1+ (-\bar L_0 e^x+\a x)^2} + e^{-2x} (-\bar
 L_0e^x+\a x) + e^{-2x} \frac{X-\bar L_0e^x+ \a }
 {1+ (-\bar L_0e^x+\a x)^2}, \\
 b &\ge& c e^{-y} \sqrt{1+ (\bar L_0e^y+\a y)^2} - e^{-2y} (\bar
 L_0e^y+\a y) + e^{-2y}   \frac{Y-\bar L_0e^y-\a} {1+ (\bar L_0e^x+\a 
y)^2}, \\
 a-b &=& \frac{\eta}{(T-t)^2}, \qquad
 \left[ \begin{array}{ll} X & 0 \\ 0 & -Y \end{array} \right] \le 0 .
\end{eqnarray*}
Subtracting the viscosity inequalities and using the last line yield
\begin{align*}
\frac{\eta}{T^2} \le&  c e^{-x} \sqrt{1+ (\bar L_0e^x-\a x)^2}-c e^{-y} \sqrt{1+
(\bar L_0e^y+ \a y)^2}+ \a e^{-2x} (x+1)+  \a e^{-2y}
(y+1) \\
&-\bar L_0e^{-x}+\bar L_0e^{-y}-\frac {\bar L_0 e^{-x}}  {1+
 (\bar L_0e^x-\a x)^2}+ \frac{\bar L_0e^y} {1+ (\bar L_0e^x+\a y)^2}
\end{align*}

Using the fact that the functions $z\mapsto \sqrt{1+z^2}$ and $z\mapsto
\frac 1 {1+z^2}$ are $1$-Lipschitz, we deduce that
\begin{align*}
\frac{\eta}{T^2} \le&  ce^{-x} \sqrt{1+ \bar L_0^2 e^{2x}}-ce^{-y} \sqrt{1+
\bar L_0^2 e^{2y}}+ e^{-x} \a(({\color{black} |c|}+\bar L_0)|x|+x e^{-x}+e^{-x})\\
&+ \a e^{-y}
(({\color{black} |c|}+\bar L_0)|y|+ y e^{-y} +e^{-y})
-\bar L_0e^{-x}+\bar L_0e^{-y}-\frac {\bar L_0 e^{-x}}  {1+
 (\bar L_0e^x)^2}+ \frac{\bar L_0e^y} {1+ (\bar L_0e^x)^2}
\end{align*}
Remarking that the function $z\mapsto e^{-z} (({\color{black} |c|}+\bar L_0)|z|+z
e^{-z}+e^{-z})$ is bounded from above by a constant $C_3$, we have
\begin{equation}\label{eq:g}
\frac{\eta}{T^2} \le 2 C_3 \a+ g(x)-g(y)
\end{equation}
where
$$g(x)=e^{-x} c \sqrt{1+ \bar L_0^2e^{2x}} -\bar L_0e^{-x}-\frac {\bar L_0 
e^{-x}}  {1+  \bar L_0^2e^{2x}}.
$$
{\color{black} \noindent {\bf Case A: $c\ge 0$}}\\
We now rewrite $g$ in the following way
\begin{align*}
g(x)=&c\sqrt{e^{-2x}+\bar L_0^2}-\bar L_0 e^{-x} -\frac {\bar L_0}{e^x+\bar
 L_0^2 e^{3x}}\\
=&\frac{c^2 e^{-2x}(1-\bar L_0^2)-\bar L_0^2e^{-2x}}{c\sqrt{e^{-2x}+\bar
   L_0^2}+\bar L_0 e^{-x}}-\frac {\bar L_0}{e^x+\bar
 L_0^2 e^{3x}}\\
=&\frac{c^2(1-\bar L_0^2)}{c\sqrt{e^{2x}+\bar L_0^2e^{4x}}+\bar L_0
 e^{x}}-\frac{\bar L_0^2}{\sqrt{e^{2x}+\bar L_0^2e^{4x}}+\bar L_0
 e^{x}}-\frac {\bar L_0}{e^x+\bar
 L_0^2 e^{3x}}
\end{align*}
and use the fact that $\bar L_0\ge 1$ to deduce that $g$ is
non-decreasing. Hence, we finally get
$$
\frac{\eta}{T^2} \le 2 C_3\a
$$
which is absurd for $\a$ small enough.
\bigskip

In order to prove that ${\color{black} M^1} \le 0$, we proceed as before and we obtain
\eqref{eq:g} where
{\color{black} $$g(x)=c  \sqrt{e^{-2x}+  L_1^2} +L_1 e^{-x}+ \frac {L_1}  {e^{x}+  L_1^2 e^{3x}}.
$$}
Remarking that $g$ is decreasing permits us to conclude in this case.\\
{\color{black} \noindent {\bf Case B: $c\le 0$}\\
We simply notice that the equation is not changed if we change $(w,c)$ in $(-w,-c)$.
}
\bigskip

{\color{black} \noindent {\bf Step 2: Lipschitz in time estimates}}\\
It remains to prove that $\bar{u}$ is $\bar{C}$-Lipschitz continuous
with respect to $t$ under the additional compatibility
condition~\eqref{cond:compatibilite}. To do so, we fix $h>0$ and we
consider the following functions:
$$
\bar{u}_h (t,r) = \bar{u} (t+h,r) - \bar{C} h \quad \text{ and }
\bar{u}^h (t,r) = \bar{u}(t+h,r) + \bar{C}h.
$$
Remark that $\bar{u}_h$ and $\bar{u}^h$ satisfy \eqref{eq::4a}. Moreover,
Proposition~\ref{prop:barriers} implies that
$$
\bar{u}_h(0,r) \le \bar{u}_0 (r) \le \bar{u}^h (0,r).
$$
Thanks to the comparison principle, we conclude that $\bar{u}_h \le
\bar{u} \le \bar{u}^h$ in {\color{black} $[0,+\infty)\times (0,+\infty)$}; since $h$ is arbitrary, we
thus conclude that $\bar{u}$ is $\bar{C}$-Lipschitz continuous with
respect to $t$. The proof of Proposition~\ref{prop:gradestim} is now
complete.
\end{proof}

\subsection{Proof of Theorem~\ref{th:cauchy}} \label{subsec:class}

It is now easy to derive Theorem~\ref{th:cauchy} from
Propositions~\ref{pro:cauchy} and \ref{prop:gradestim}.
\begin{proof}[Proof of Theorem~\ref{th:cauchy}]
 Consider the viscosity solution $\bar{u}$ given by
 Proposition~\ref{pro:cauchy}
{\color{black}  with $f(t)=\bar C t$ where the constant $\bar C$ 
is given in the barrier presented in Proposition \ref{prop:barriers}.}

 This function is
 continuous. Moreover, thanks to Proposition~\ref{prop:gradestim},
 $\bar{u}_t$ and $\bar{u}_r$ are bounded in the viscosity sense;
 hence $u$ is Lipschitz continuous. In particular, there exists a set
 $\tilde{N} \subset {\color{black} (0,+\infty)} \times (0,+\infty)$ of null measure such
 that for all $(t,r) \notin \tilde{N}$, $\bar{u}$ is differentiable at
 $(t,r)$.

 Thanks to the equation
\begin{equation}\label{eq::vsd}
\displaystyle{\bar{u}_t-a(r,\bar{u}_r)\bar{u}_{rr}=f(r,\bar{u}_r)}
\quad \mbox{for}\quad (t,r)\in (0,+\infty)\times (0,+\infty)
\end{equation}
 with
$$
a(r,\bar{u}_r)=\frac{1}{1 + r^2\bar{u}_r^2}, \quad
f(r,\bar{u}_r)=\frac{1}{r}\left\{c \sqrt{1 + r^2\bar{u}_r^2} + \bar{u}_r
 \left(\frac{2+ r^2 \bar{u}_r^2}{1+ r^2 \bar{u}_r^2}\right)\right\}
$$ 
we also have that $\bar{u}_{rr}$ is locally bounded in the viscosity
sense. This implies that $\bar{u}$ is locally $C^{1,1}$ with respect
to $r$, and in particular, we derive from Alexandrov's theorem
\cite[p.~242]{eg92} that for all $t>0$ there exists a set $N_t \subset
[0,+\infty)$ of null measure such that for all $r \notin N_t$,
$\bar{u}(t,\cdot)$ is twice differentiable with respect to $r$,
i.e. there exist $p, A \in \R$ such that for $\rho$ in a neighborhood
of $r$, we have
\begin{equation}\label{eq:dl2}
 \bar{u} (t,\rho) = \bar{u} (t,r) + p (\rho-r) + \frac12 A (\rho-r)^2 + o ((\rho-r)^2). 
\end{equation}
From $\tilde{N}$ and $\{ N_t \}_{t >0}$, we can construct a set $N
\subset (0;+\infty) \times (0;+\infty)$ of null measure such that for
all $(t,r) \notin N$, $\bar{u}$ is differentiable with respect to time
and space at $(t,r)$ and there exists $A \in \R$ such that
\eqref{eq:dl2} holds true.  We conclude that
$$
\bar{u} (s,\rho) = \bar{u} (t,r) + \partial_t \bar{u} (t,r) (s-t)
+  \partial_r \bar{u}(t,r) (\rho-r) + \frac12 A (\rho-r)^2 + o ((\rho-r)^2) + o (s-t).
$$
In particular, \eqref{eq::vsd} holds true for $(t,r) \notin N$.

We deduce from the previous discussion that $\bar{u}_t - \bar{u}_{rr}
=\tilde{f} \in L^\infty_{loc}$ holds true almost everywhere, and thus
in the sense of distributions. From the standard interior estimates
for parabolic equations, we get that $\bar{u} \in W^{2,1;p}_{loc}$ for
any $1<p<+\infty$. Then from the Sobolev embedding (see Lemma 3.3 in
\cite{LSU}), we get that for $p>3$, and $\alpha=1-3/p$, we have
$\bar{u}_r \in C^{\alpha,\alpha/2}_{loc}$.

We now use that \eqref{eq::vsd} holds almost everywhere.
Therefore we can apply the standard interior Schauder theory (in
H\"{o}lder spaces) for parabolic equations.  This shows that $\bar
u\in C^{2+\alpha,1+\alpha/2}_{loc}$. Bootstrapping, we finally get
that $\bar u\in C^\infty_{loc}$, which ends the proof of the theorem.
\end{proof}

\section{Construction of a {\color{black} general} weak (viscosity) solution}
\label{sec:withoutcompatibility}

{\color{black} The main goal of this section is to prove Theorem \ref{th:cauchywc}.
We start with general barriers, H\"{o}lder estimates in time 
and finally an approximation argument.}
\begin{pro}[Barriers for the Cauchy problem without the Compatibility
 Condition] \label{prop:barriers2} Let $\bar{u}_0 \in {\color{black} W^{2,\infty}_{loc}(0,+\infty)}$ 
be such that there exists $C_0$ such that
\begin{equation}\label{eq:hypu0}
|\buor|\le C_0 \quad{\rm and}\quad|\kappa_{\buo}|\le C_0.
\end{equation}
Then, there exists a constant $\bar{C}>0$ (depending only on $C_0$)
such that for any function $B:[0,T]\to \R$ with $B(0)=0$  
and $B'\ge \bar C(1+\bar Ct)$,  $\bar{u}^\pm
(t,r)=\bar{u}_0 (r) \pm \frac{\bar C t}r \pm B(t)$ are respectively a
super- and a sub-solution of \eqref{eq::4a},\eqref{eq::4b}.
\end{pro}

\begin{proof}
We only do the proof for the super-solution since it is similar
(and even simpler) for the sub-solution. {\color{black} We also do the proof only in the case $c\ge 0$, 
noticing that the equation is unchanged if we replace $(w,c)$ with $(-w,-c)$.}

It is convenient to write $A$ for $\bar C t$ and do the
computations with this function.
Since $|\kappa_{\buo}|\le C_0$, we have
$$\left|\frac {r(\buo)_{rr}}{(1+(r\buor)^2)^{\frac 32}} + \buor \left(\frac
   {2+(r\buor)^2}{(1+(r\buor)^2)^{\frac 32}}\right)\right|\le C_0.
$$
{\color{black} Since $|u_r|\le C_0$}, there exists $c_1>0$ such that
$$\left|r(\buo)_{rr}\right|\le c_1(1+(r\buor)^2)^{\frac 32}.$$

We then have 
\begin{align*}
 \bar F(r,\bar u^+_r, \bar u^+_{rr})=&c\sqrt{1+(r \bar u^+_r)^2} +
 \bar u_r
 \left(\frac {2+ (r\bar u_r^+)^2}{1+(r\bar u_r^+)^2}\right)+\frac {r
   \bar u_{rr}^+}{1+(r\bar u_r^+)^2}\\
 =&c\sqrt{1+\left(\frac {-A}r+r\buor\right)^2} +
 \left(\frac{-A}{r^2}+\buor\right)\left(\frac {2+\left(\frac
       {-A}r+r\buor\right)^2}{1+\left(\frac
       {-A}r+r\buor\right)^2}\right)\\
 &+ \frac{\frac {2A}{r^2}+r(\buo)_{rr}}{1+\left(\frac
     {-A}r+r\buor\right)^2}\\
 \le &c( 1+ \frac A r + r|\buor|) + \buor \left(\frac {2+\left(\frac
       {-A}r+r\buor\right)^2}{1+\left(\frac
       {-A}r+r\buor\right)^2}\right) \\
 &-\frac A {r^2} \left(\frac {\left(\frac
       {-A}r+r\buor\right)^2}{1+\left(\frac
       {-A}r+r\buor\right)^2}\right) + c_1 \frac
 {(1+(r\buor)^2)^{\frac 32}}{1+\left(\frac {-A}r+r\buor\right)^2}.
\end{align*}
Using \eqref{eq:hypu0}, we can write
$$\buor \frac {2+\left(\frac
       {-A}r+r\buor\right)^2}{1+\left(\frac
       {-A}r+r\buor\right)^2}\le 2C_0$$
we get
\begin{align*}
 \bar F(r,\bar u^+_r, \bar u^+_{rr})\le& c(1+ \frac A r +C_0r) + 2 C_0 -\frac
 A {r^2} \left(\frac {\left(\frac
       {-A}r+r\buor\right)^2}{1+\left(\frac
       {-A}r+r\buor\right)^2}\right) + c_1 \frac
 {(1+(r\buor)^2)^{\frac 32}}{1+\left(\frac {-A}r+r\buor\right)^2}\\
\le& c(1+ \frac A r +C_0r) + 2C_0 -\frac
 A {r^2} \left(\frac {\left(\frac
       {-A}r+r\buor\right)^2}{1+\left(\frac
       {-A}r+r\buor\right)^2}\right) + c_1 \frac
 {(1+r|\buor|)^{3}}{1+\left(\frac {-A}r+r\buor\right)^2}.
\end{align*}

We now set $\rho$ such that $r\buor=\rho\frac A r$ and distinguish two
cases:
\medskip

\noindent{\bf Case 1: $\frac 1 2 <\rho <2$.} In this case, 
\begin{align*}
\bar F(r,\bar u^+_r, \bar u^+_{rr})\le& c(1+ 2 C_0r +C_0r) +2C_0 + c_1\left(
 1+r|\buor|\right)^3\\
\le&c+ 3 c C_0 r+2C_0+ 4 c_1+ 4c_1r^3|\buor|^3\\
\le&c+ 3 c C_0 r+2C_0+ 4 c_1+ 4c_1 \rho \frac A r r^2C_0^2 \\
\le&(c+2C_0 + 4 c_1)+ r( 3 cC_0 + 8c_1 C_0^2 A)  \\
\end{align*}
where for the second line, we have used the fact that for $a,b\ge 0$,
$(a+b)^3\le 4 (a^3+b^3)$. On the other hand, we have $r\bar u_t^+=A'+rB'$.
Choosing $\bar C\ge \max (c+2C_0 + 4 c_1, 3 C_0+ 8c_1 C_0^2)$ we
get the desired result in this case.
\bigskip

\noindent{\bf Case 2: $\rho\le \frac 1 2$ or $\rho \ge2$.} In this case 
$$
\frac {(1+r|\buor|)^{3}}{1+\left(\frac {-A}r+r\buor\right)^2}\le\frac{
 4+4 r^3|\buor |^3}{1+(\rho-1)^2\frac {A^2}{r^2}} \le 4 + 4 \frac
{\rho^2r|\buor|}{(\rho-1)^2}\le 4+ 16 C_0 r.
$$
Then 
$$ 
\bar F(r,\bar u^+_r, \bar u^+_{rr})\le c+ c\frac A r +2C_0 + 4c_1+
c C_0r +16c_1C_0r -\frac A {r^2} \frac {(\rho-1)^2 \left(\frac
   {A}r\right)^2}{1+(\rho-1)^2\left(\frac {A}r\right)^2}
$$
We distinguish two sub-cases:
\medskip

\noindent{\sc Subcase 2.1: $\frac A r \le 2$}.
In this sub-case, we get 
$$ 
\bar F(r,\bar u^+_r, \bar u^+_{rr})\le(3c+ 2C_0+4c_1)+r(cC_0+16c_1C_0)
$$
and we obtain the desired result taking $\bar C\ge\max(3c+
2C_0+4c_1,cC_0+16c_1C_0)$.

\medskip
\noindent{\sc Subcase 2.2: $\frac A r \ge 2$}.
In this subcase, $|\rho-1|\frac A r\ge 1$ and
$$\frac {(\rho-1)^2 \left(\frac
       {A}r\right)^2}{1+(\rho-1)^2\left(\frac
       {A}r\right)^2}\ge \frac 1 2 $$
and thus
\begin{align*}
 \bar F(r,\bar u^+_r, \bar u^+_{rr})\le& (c+2 C_0 + 4c_1) + A(\frac c
 r- \frac 1 {2r^2})  +cC_0r+16c_1C_0r  \\
 \le &(c+2C_0+ 4c_1)+ (dA +cC_0+16c_1C_0)r
\end{align*}
where for the last line we have used the fact that we can find
$d >0$ (only depending on $c$) such that $\frac c r-
\frac {1}{2r^2} \le d r$ for all $r >0$.  We finally get the desired
result taking $\bar C\ge\max(c+2C_0+ 4c_1,cC_0+16c_1C_0,d)$.  The
proof is now complete.
\end{proof}

\begin{pro}[Time H\"older estimate -- (I)]\label{pro:holder1}
 Let $\buo\in {\color{black} W^{2,\infty}_{loc}(0,+\infty)}$ satisfying
 \eqref{eq:hypu0}. Let $\bar u$ be a solution of
 \eqref{eq::4a},\eqref{eq::4b} satisfying \eqref{baru:baru0sst}. If
 $\bar u$ is $L_0$-Lipschitz continuous with respect to the variable
 $r$, then there exists a constant $C$, depending only on $C_0$ and
 $L_0$ such that
$$
|\bar u(t,r)-\buo(r)|\le C\sqrt t + B(t)
$$
where $B$ is defined in Proposition \ref{prop:barriers2}.
\end{pro}
\begin{rem}\label{rem:holder}
Let us note that in Proposition \ref{prop:barriers2}, we can choose
$B(t)=\bar C t (1+\frac{\bar C}2 t)$. Hence,  we deduce from
Proposition \ref{pro:holder1} that there exists
$C>0$ such that for all $t \in [0,1]$, 
\begin{equation}\label{eq::rm3}
|\bar u(t,r)-\buo(r)|\le C\sqrt t.
\end{equation}

\end{rem}
\begin{proof}
 Let $r_0>0$. Using Proposition \ref{prop:barriers2} and the
 comparison principle, we deduce that there exists a
 constant $\bar C$ and a function $B$ such that
$$\left| \bar u(t,r_0)-\buo(r_0)\right|\le \bar C
\frac t{r_0} +B(t).$$
Since $\bar u$ is $L_0$-Lipschitz continuous in $r$, we also have
$$
|\bar u(t,0)-\bar u(t,r_0)|\le L_0r_0\quad{\rm and}\quad
|\buo(0)-\buo(r_0)|\le C_0 r_0.
$$

Combining the previous inequalities, we get that
$$
|\bar u(t,0)-\buo(0)|\le (L_0+C_0) r_0+ \bar C \frac t {r_0} + B(t).
$$
Taking the minimum over $r_0$ in the right hand side, we get that
$$
|\bar u(t,0)-\buo(0)|\le C_1 \sqrt t + B(t)
$$
with $C_1:=2\sqrt{(C_0+L_0)\bar C}$.

We finally deduce that
$$|\bar u(t,r)-\buo(r)|\le \min\left\{\bar C\frac t r +  B(t), C_1
 \sqrt t + B(t) +(L_0+C_0) r\right\}.$$ 
The desired result is obtained by remarking that, if $r\le \sqrt t$,
then   $C_1 \sqrt t + B(t) +(L_0+C_0) r\le (C_1+L_0+C_0)\sqrt t+ B(t)$, while
if $r\ge \sqrt t$, then $\bar C\frac t r +  B(t)\le \bar C \sqrt t+ B(t)$.
\end{proof}

The next proposition asserts that the previous proposition is still
true if we do not assume that $\bar u$ is Lipschitz continuous with
respect to $r$.
\begin{pro}[{\color{black} Existence and time} H\"older estimate -- (II)]\label{pro:holder2}
 Let $\buo\in {\color{black} W^{2,\infty}_{loc}(0,+\infty)}$ satisfying
 \eqref{eq:hypu0}. {\color{black} Then there exists a solution $\bar u$ of
 \eqref{eq::4a},\eqref{eq::4b} satisfying \eqref{baru:baru0sst}. Moreover
 there exists a constant $C$, depending only on $C_0$ such
 that}
$$|\bar u(t,r)-\buo(r)|\le C\sqrt t + B(t)$$
where $B$ is defined in Proposition \ref{prop:barriers2},
{\color{black}and there exists a constant $L_0$ (only depending on $C_0$) such that
$$|\bar u(t,r+\rho)-\bar u(t,r)|\le L_0|\rho|.$$}
\end{pro}
\begin{proof}
 The initial datum is approximated with a sequence of initial data
 satisfying \eqref{eq:hypu0} and the compatibility
 condition~\eqref{cond:compatibilite}; passing to the limit will
give the desired result.

We can assume without loss of generality that $C_0 \ge \frac{c}2$. 
Then we consider 
$$
\buo^\e=\Psi_\e U_0+(1-\Psi_\e)\buo
$$
where $U_0\in C^\infty$ is such that 
\begin{equation}\label{eq::rm5}
U_0(0)=\buo(0), \quad (U_0)_r(0)=-\frac c2,\quad |(U_0)_r|\le C_0,\quad
r|(U_0)_{rr}|\le C_0 \quad {\rm for}\quad  r\le 2
\end{equation}
and 
$$
\Psi_\e(r)=\Psi_1\left(\frac r \e\right)
$$
where the non-increasing function $\Psi_1\in C^\infty$ satisfies
$$
\Psi_1=\left\{\begin{array}{lll}
1&{\rm if}& r\le 1,\\
0&{\rm if}& r\ge 2.
\end{array}\right.
$$
\begin{claim}\label{claim}
 The initial condition $\buo^\e$ satisfies the compatibility
 condition~\eqref{cond:compatibilite} and \eqref{eq:hypu0} for some
 constant $C_0$ which does not depend on $\e$.
\end{claim}
Let $u^\e$ denote the unique solution of \eqref{eq::4a} with initial
condition $\buo^\e$ given by Proposition~\ref{pro:cauchy}, 
{\color{black} using the barrier (Proposition \ref{prop:barriers2}) provided by the Claim \ref{claim}.}
In
particular, $u^\e$ satisfies \eqref{baru:baru0sst} for some constant $\bar
C^\e$ depending on $\e$. Using Proposition~\ref{prop:gradestim}, we
deduce that $\bar u^\e$ is $L_0$-Lipschitz continuous with
$L_0:=\max(1,C_0)$. Then Proposition~\ref{pro:holder1} can be applied
to obtain the existence of a constant $C$ (depending only on $C_0$, 
{\color{black} because $L_0$ now depends on $C_0$}) such
that for all $\e$
$$|\bar u^\e(t,r)-\buo^\e(r)|\le C \sqrt t + B(t).$$ 
Taking $\e\to 0$ and using the stability of the solution and the
uniqueness of \eqref{eq::4a},\eqref{eq::4b}, we finally deduce the
desired result.
\end{proof}
We now prove the claim.
\begin{proof}[Proof of Claim~\ref{claim}.]
We have
$$(\buo^\e)_r=(\Psi_\e)_r (U_0-\buo)+\Psi_\e(U_0)_r+(1-\Psi_\e)(u_0)_r.
$$

Hence, since $(\Psi_\e)_r(0)=0$ and $\Psi_\e(0)=1$, we get
$$(\buo^\e)_r(0)=(U_0)_r(0)=-\frac c2$$
which means that $\buo^\e$ satisfies \eqref{cond:compatibilite1}. 
{\color{black} Using the fact that 
$\buo^\e\in W^{2,\infty}_{loc}$
and (\ref{eq::rm5}), we get \eqref{cond:compatibilite}.}

Since $U_0(0)=\buo(0)$ and $U_0$ and $\buo$ are $C_0$-Lipschitz
continuous, we have
$$
|U_0(r)-\buo(r)|\le 2 C_0 r.
$$
Let $c_1$ denote $ \sup_{\rho\ge 0} \rho
|(\Psi_1)_r(\rho)|<+\infty$. We then have
$$
|(\Psi_\e)_r (U_0-\buo)|\le 2 C_0 \frac r \e
\left|(\Psi_1)_r\left(\frac r\e\right)\right|  \le 2C_0 c_1.
$$
Hence 
$$
|(\buo^\e)_r|\le 2C_0(c_1+1).
$$

Let us now obtain an estimate on $\kappa_{\buo^\e}$. Using the previous
bound, we only have to estimate 
$$
\frac
{r(\buo^\e)_{rr}}{\left(1+(r(\buo^\e)_r)^2\right)^{\frac 32}}.
$$
If $r >2$, then $\buo^\e=\buo$ and the estimate follows from \eqref{eq:hypu0}.
If $r\le 2$, it is enough to estimate $r(\buo^\e)_{rr}$. We have
$$
r(\buo^\e)_{rr}=r(\Psi_\e)_{rr}(U_0-\buo)+2r(\Psi_\e)_r((U_0)_r-\buor)+r\Psi_\e
(U_0)_{rr}+r(1-\Psi_\e)(\bar u_0)_{rr}.
$$
Moreover there exists a constant $c_2$ (depending only on $C_0$) such
that for all $r \le 2$, $r|(\bar u_0)_{rr}|\le c_2$. Let $c_3$ denote
$\sup_{\rho\ge 0} \rho ^2|(\Psi_1)_{rr}(\rho)|<+\infty$. We then
have
$$
r|(\Psi_\e)_{rr} (U_0-\buo)|\le 2 C_0 \frac {r^2} {\e^2}
\left|(\Psi_1)_{rr}\left(\frac r\e\right)\right|  \le 2C_0 c_3.
$$
We finally deduce that for $r\le 2$,
$$
|r(\buo^\e)_{rr}|\le 2C_0c_3+4C_0c_1+C_0+c_2
$$
which proves that $\buo^\e$ satisfies \eqref{eq:hypu0} with a constant
$\bar C_0=2C_0c_3+4C_0c_1+C_0+c_2$ depending only on $C_0$. 
\end{proof}

We now turn to the proof of Theorem \ref{th:cauchywc}.

\begin{proof}[Proof of Theorem \ref{th:cauchywc}]
 {\color{black} The existence of $\bar u$ and its Lipschitz
 continuity with respect to $r$ follows from
 Proposition~\ref{pro:holder2}.
The uniqueness (and continuity) of $\bar u$ 
follows from the comparison principle
 (Theorem~\ref{thm comp spirales sous lin en r}).}
 Let us now prove that $\bar u$ is $\frac 12$-H\"older continuous
 with respect to time. By Remark \ref{rem:holder}, there exists a
 constant $C$ such that for $h\le 1$
$$|\bar u(h,r)-\buo(r)|\le C \sqrt h.$$
{\color{black} with $C$ given in (\ref{eq::rm3}).
Proceeding as in Step 2 of the proof of Proposition \ref{prop:gradestim}, we get for $0\le h\le 1$:
$$\bar u(t+h, r)-\bar u(t,r)\le C\sqrt h.$$
The reverse inequality is obtained in the same way. 
This implies (\ref{eq::rm1}). The proof is now complete.}
\end{proof}

\appendix

\section{Appendix: proofs of technical lemmas}
\label{app:a}
\renewcommand{\theequation}{A.\arabic{equation}}
\setcounter{theo}{0}
\setcounter{equation}{0}
\renewcommand{\thesubsection}{A.\arabic{subsection}}

\begin{proof}[Proof of Lemma \ref{lem:psi}]

We look for $\psi$ under the following form: for $x,\theta \in \R$,
$$
\psi (x,\theta) = (1-\iota (x)) (x,e^{i\theta}) + \iota (x) (0,e^{x+i\theta})
$$
where $\iota:\R \to \R$ is {\color{black} non-decreasing}, smooth ($C^\infty$) and such that
$\iota(x)=0$
if $x \le 0$ and $\iota(x) = 1$ if $x \ge 1$. Remark that \eqref{eq:ass-psi}
and \eqref{psi:add} are readily satisfied.

It remains to prove \eqref{psi:grad-inf} and \eqref{psi:grad-inf-bis}.
Let us first find $\eps>0$ and $m_\psi>0$ such that for all
$x,y,\theta,\sigma$ such that $|(x,\theta)-(y,\sigma)| \le \eps$,
we have \eqref{psi:grad-inf} and
\eqref{psi:grad-inf-bis}.

\paragraph{Study of \eqref{psi:grad-inf}.} It is convenient to use the
following notation: $\psi (x,\theta) = (\phi_1 (x), \phi_2 (x)
e^{i\theta})$.  We first write \eqref{psi:grad-inf} in terms of
$\phi_i$:
$$
 |\phi_1 (x) - \phi_1 (y) | + |\phi_2 (x) - \phi_2 (y) \cos
 (\theta-\sigma)|
 + \phi_2 (y) |\sin (\theta - \sigma)|
 \ge m_\psi (|x-y| + |\theta-\sigma|)
$$
(we used a different norm in $\R^3$ and $m_\psi$ is changed accordingly).
It is enough to prove
\begin{multline*}
 |\phi_1 (x) - \phi_1 (y) | + |\phi_2 (x) - \phi_2 (y) |
 + \phi_2 (y) (|\sin| - 1 + \cos) (\theta - \sigma) \\
 \ge m_\psi (|x-y| + |\theta-\sigma|).
\end{multline*}
We choose $\eps \le 1$ and we remark that such an inequality
is clear if $x \le -1$ or $x \ge 2$.
Through a Taylor expansion and using the fact that $\phi_2 (y) \ge 1$, this
reduces to check that
$$
\min (\inf_{x \in (-1,2)} (|\phi_1'(x)|+ |\phi_2'(x)|), 1) \ge 2 m_\psi
$$
which reduces to
$$
\inf_{x \in {\color{black} (-1,2)}} \{ |\phi_1'(x)|+ |\phi_2'(x)| \} > 0.
$$
For $x$ far from $0$, a simple computation shows that $\phi_2'(x) \ge
\iota (x)e^x$ {\color{black} (for $x\ge 0$)}
and this permits us to conclude. For $x$ in a neighborhood of $0$,
$\phi_1'(x) = 1 +o(1)$
and $\phi_2'(x) = O (x)$ and we can conclude in this case too.
{\color{black} In $[-1,2]\backslash [0,1]$, the conclusion is straightforward.}

\paragraph{Study of \eqref{psi:grad-inf-bis}.}
We next write \eqref{psi:grad-inf-bis} in terms of $\phi_i$
\begin{multline}\label{app:eq1}
|\Phi (x,y) + \phi'_2 (x) \phi_2 (y) (1 -\cos (\theta -\sigma))|
+ |\phi_2 (x) \phi_2 (y)| |\sin (\theta-\sigma)| \\
\ge {\color{black} m_\psi} (|x-y| + |\theta -\sigma|)
\end{multline}
where
$$
\Phi (x,y) = \phi_1'(x) (\phi_1 (x) - \phi_1 (y)) + \phi_2'(x) ( \phi_2
(x) - \phi_2 (y)).
$$
Once again, {\color{black} the previous inequality} is true for $x \notin (-1,2)$
and for $x \in (-1,2)$, we choose $m_\psi$ such that
$$
\inf_{x \in (0,1)} \{ {\color{black} (\phi_1'(x))^2 + (\phi_2'(x))^2} \} \ge 2 m_\psi.
$$
The same reasoning as above applies here too.

\paragraph{Reduction to the case: $|(x,\theta) -(y,\sigma)|\le \eps$.}
It remains to prove that for $\eps>0$ given, we can find $\delta_0>0$
such that, as soon as $|\psi(x,\theta) -\psi(y,\sigma)| \le \delta_0$
and $|\theta-\sigma|\le \frac\pi2$, then $|(x,\theta)-(y,\sigma)|\le
\eps$. We argue by contradiction by assuming that there exists
$\eps_0>0$ and two sequences $(x_n,\theta_n)$ and $(y_n,\sigma_n)$
such that
\begin{eqnarray*}
|\theta_n - \sigma_n | \le \frac\pi2 \\
|x_n-y_n|+ |\theta_n -\sigma_n| \ge \eps_0 \\
\phi_1 (x_n) - \phi_1 (y_n)  \to 0 \\
\cos(\theta_n -\sigma_n) \phi_2 (x_n)- \phi_2 (y_n) \to 0 \\
\phi_2 (x_n) \sin (\theta_n-\sigma_n) \to 0
\end{eqnarray*}
as $n \to \infty$. Since $\phi_2$ is bounded from below by $1$, we
deduce that $\sin (\theta_n -\sigma_n) \to 0$. Up to a subsequence, we
can assume that $\theta_n - \sigma_n \to \delta$ and we thus deduce
that $\delta = 0$. Hence, $|x_n -y_n| \ge \frac{\eps_0}2$ for large
$n$'s.  Thanks to a Taylor expansion in $\theta_n - \sigma_n$, we can
also get that $\phi_2 (x_n) - \phi_2 (y_n) \to 0$. Because
$|x_n-y_n|\ge \frac{\eps_0}2$, we then get that $x_n$ and $y_n$ remain
in a bounded interval. We can thus assume that $x_n \to x_*$ and $y_n
\to y_*$. Finally, we have $\phi_i (x_*) = \phi_i (y_*)$ for $i=1,2$
and {\color{black} $|x_*-y_*|\ge \frac{\eps_0}2$} which is impossible. The proof of
the lemma is now complete.

\end{proof}

\begin{proof}[Proof of Lemma \ref{lem:est}]
The second estimate is satisfied if $C_2 $ is chosen such that
$$
C_2 \ge \sup_{r>0}  \left(r - \left(r-\frac \pi 3\right)_+^2\right) . 
$$
We now prove the first estimate. We distinguish three cases:

\paragraph{Case 1: $x\le 1$ and $y\le 1$.}
In this case, $e^x$ and $e^y$ are bounded and the definition of $u_0$ 
in terms of the Lipschitz continuous function $\bar u_0$ implies
$$
|u_0(x)-u_0(y)|\le C
$$
for some constant $C>0$. 

\paragraph{Case 2: ($x\le 1$ and $y\ge 1$) or ($x\ge 1$ and $y\le 1$).}
The two cases can be treated similarly and we assume here that $x\le
1$ and $y\ge 1$. 
In that case $\psi(x,\theta)=(a,b)$ with $a\in \R$ and $b\in \C$
with $|b|\le e$ (see \eqref{psi:add}) and
$\psi(y,\sigma)=(0,e^{y+i\sigma}).$ Moreover, there exists a
constant $C$ such that
$$|u_0(x)-u_0(y)|\le C(1+e^y).$$
We also have
\begin{align*}
|\psi(x,\theta)-\psi(y,\sigma)|= & \sqrt{a^2
 +|e^{y+i\sigma}-b|^2}\\
\ge&|e^{y+i\sigma}-b|\\
\ge& e^{y}-|b|\\
\ge& e^y-e.
\end{align*}
Hence,
\begin{eqnarray*}
|u_0(x)-u_0(y)|&\le & C (1+e) + C(e^y-e)\\
&\le& C (1+e)+ C^2 \eps e^{-Kt} + \frac{e^{Kt}}{4\eps} (e^y-e)^2 \\
&\le& C(1+e+C)+\frac{e^{Kt}}{4\e}|\psi(x,\theta)-\psi(y,\sigma)|^2
\end{eqnarray*}
which gives the desired estimate.

\paragraph{Case 3: $x\ge 1$ and $y\ge 1$.}
In this case,
$$|\psi(x,\theta)-\psi(y,\sigma)|=|e^{x+i\theta}-e^{y+i\sigma}|\ge
 |e^x-e^y|$$
and
$$|u_0(x)-u_0(y)|\le L_{u_0} |e^x-e^y|,$$
where $L_{u_0}$ is the Lipschitz constant of $\bar u_0$.
Hence, $C_2$ is chosen such that 
$$
C_2 \ge \sup_{r>0} {\color{black} \left(L_{u_0} r - \frac{1}{4\eps} r^2\right)}.
$$
The proof is now complete.
\end{proof}

\paragraph{Acknowledgements.} The authors would like to thank Guy Barles
for fruitful discussions during the preparation of this article.

\def\cprime{$'$}

\end{document}